\documentclass[11pt,a4paper]{article}

\usepackage[math]{cellspace}
\usepackage{psfrag}
\usepackage{latexsym}
\usepackage{lscape}
\usepackage{afterpage}
\usepackage[color=green!40]{todonotes}
\usepackage{amsmath,amssymb,amsthm,graphicx,epsfig,float,cite,rotating,subcaption,comment,bigints,hyperref,multirow}
\usepackage{arydshln}

\hypersetup{
    colorlinks=true,
    linkcolor=blue,
    filecolor=magenta,      
    urlcolor=cyan,
    citecolor=red,
}
\DeclareMathOperator*{\argmin}{argmin}

\newcommand{\Id}{\ensuremath{\operatorname{Id}}}

\newcommand{\tb}{\widetilde{\beta}}

\newcommand{\ds}{\displaystyle}
\newcommand{\nexto}{\kern -0.54em}
\newcommand{\dR}{\mathbb{R}}

\newcommand{\dZ}{{\cal Z \kern -0.7em Z}}
\newcommand{\dC}{{\rm\hbox{C \kern-0.8em\raise0.2ex\hbox{\vrule
height5.4pt width0.7pt}}}}
\newcommand{\dQ}{{\rm\hbox{Q \kern-0.85em\raise0.25ex\hbox{\vrule
height5.4pt width0.7pt}}}}
\newcommand{\proofbox}{\hspace{\fill}{$\Box$}}
\newtheorem{lemma}{Lemma}
\newtheorem{theorem}{Theorem}
\newtheorem{corollary}{Corollary}

\newtheorem{remark}{Remark}

\newtheorem{definition}{Definition}
\theoremstyle{definition}
\newtheorem{algorithm}{Algorithm}

\begin{document}

\title{\vspace{-10mm}\bf Douglas--Rachford algorithm for control-constrained minimum-energy control problems}

\author{
Regina S. Burachik\thanks{Mathematics, UniSA STEM, University of South Australia, Mawson Lakes, S.A. 5095, Australia. Emails:~regina.burachik@unisa.edu.au, bethany.caldwell@mymail.unisa.edu.au, yalcin.kaya@unisa.edu.au.}
\qquad
Bethany I. Caldwell\footnotemark[1]
\qquad
C. Yal{\c c}{\i}n Kaya\footnotemark[1]
}

\maketitle

\begin{abstract} {\noindent\sf  
Splitting and projection-type algorithms have been applied to many optimization problems due to their simplicity and efficiency, but the application of these algorithms to optimal control is less common. In this paper we utilize the Douglas--Rachford (DR) algorithm to solve control-constrained minimum-energy optimal control problems. Instead of the traditional approach where one discretizes the problem and solves it using large-scale finite-dimensional numerical optimization techniques we split the problem in two subproblems and use the DR algorithm to find an optimal point in the intersection of the solution sets of these two subproblems hence giving a solution to the original problem. We derive general expressions for the projections and propose a numerical approach.  We obtain analytic closed-form expressions for the projectors of pure, under-, critically- and over-damped harmonic oscillators. We illustrate the working of our approach to solving not only these example problems but also a challenging machine tool manipulator problem.  Through numerical case studies, we explore and propose desirable ranges of values of an algorithmic parameter which yield smaller number of iterations.}
\end{abstract}
\begin{verse}
\end{verse}


\pagestyle{myheadings}
\markboth{}{\sf\scriptsize Douglas--Rachford Algorithm for Minimum-energy Control by R.~S.~Burachik, B.~I.~Caldwell, and C.~Y.~Kaya}

\section{Introduction}
Linear-quadratic (LQ) control problems are an important class of optimal control problems with a quadratic cost (or objective) functional to minimize subject to linear differential equation constraints describing the dynamics---see for theory and applications \cite{BurKayMaj2014,MTM,AmmKen1998,MauObe2003,Mou2011,BusMau2000,KugPes1990}. In this paper we will study applications of projection methods to solving the minimum-energy control of pure, under-, critically- and over-damped harmonic oscillators, as well as a machine tool manipulator, which are all examples of LQ control problems.  In fact, in all these applications we impose constraints on the control variable which makes the problems computationally challenging, justifying a novel implementation of projection methods.  For the quadratic objective functional, we consider the square norm of the control variable throughout the paper.  These problems are what we refer to as {\em minimum-energy control problems}\footnote{It must be stressed that we are not necessarily minimizing the ``true'' energy of for example a harmonic oscillator per se from a physics point of view. Rather, we are concerned with minimizing the ``energy of the control or signal'' or the ``energy of the force.''  Elaboration of this subtle difference in the terminology can also be found in \cite[Section~6.17]{AthaFalb1966}, \cite[Section~5.5]{Kirk1970}, \cite[Section~2.9]{Klamka2019} and \cite[page~118]{Sethi2019}.}.

Projection methods are an emerging field of research in mathematical optimization with successful applications to a wide range of problems, including road design \cite{BauKoch}, protein reconstruction \cite{AraBorTam2014}, sphere packing \cite{GraEls2008}, sudoku \cite{Bau2008}, graph colouring problems \cite{AraCampEls2020} and, radiation therapy treatment planning \cite{AltCenPow1988}.  These methods have chiefly been applied to discrete-time optimal control problems \cite{OdoStaBoy2013}, but there has been little or no research into applications to continuous-time optimal control problems, except recently in~\cite{BausBuraKaya2019} by Bauschke, Burachik and Kaya.  In~\cite{BausBuraKaya2019} various projection methods are applied to solve the energy minimizing double integrator problem, where the control variable is constrained, with promising results. The numerical experiments show that projection methods outperform a method employing direct discretization even in solving this relatively simple optimal control problem.



The aforementioned direct discretization approach is to first discretize the problem, typically using a Runge--Kutta method such as the Euler or trapezoidal methods, and then apply finite-dimensional optimization software (for example, AMPL~\cite{AMPL} paired with Ipopt~\cite{WacBie2006}) in order to solve the resulting large scale discrete-time optimal control problem. We aim to show the merits of the Douglas--Rachford (DR) algorithm, a popular projection method extended to solving optimization problems.  In particular, we aim to solve LQ control problems, which are much more general than the double integrator problem, and compare the DR algorithm with direct discretization.

The approach in this paper exploits the structure of LQ control problems to obtain advantages, just as the approach in \cite{BausBuraKaya2019} does the same with the simple double integrator problem. In our approach we split the constraints of the original LQ problem into two sets: one contains the ODE constraints involving the state variables, and the other contains box constraints on the control variables.  These sets are subsets of a Hilbert space, the first one of these subsets constituting a closed affine set (see Corollary \ref{cor: A closed}) and the second one a closed and convex set.  We define two simpler optimal control subproblems for computing projections, one subject to the affine set and the other to the box.  Solutions to these subproblems yield the projectors onto each of the two sets. 

The main contributions of this paper are as follows.
\begin{itemize}
    \item We derive a general expression for the projectors onto the affine and box sets of the minimum-energy control problem. (See Theorems~\ref{thm:projA_gen} and \ref{thm:projB}.)
    \item We obtain closed-form analytical expressions for the projectors of the special problems whose dynamics involve pure as well as under-, critically- and over-damped harmonic oscillators.  (See Corollaries~\ref{cor:projA_PHO}--\ref{cor:projA_PDHO3}, resp., for projections onto the affine sets of each case, and Corollary~\ref{cor:projB} for projection onto the box.)
    \item The projector expression in Theorem~\ref{thm:projA_gen} necessitates the knowledge of the state transition matrix as well as the Jacobian of the near-miss function of the shooting method.  For the case of general minimum-energy control, we present a computational algorithm (namely Algorithm~\ref{alg:projA}) for constructing the state transition matrix and the Jacobian and thus finding a projector onto the affine set which in turn can be used in general projection algorithms.
    \item We illustrate the working of Algorithm~\ref{alg:projA} and Theorem~\ref{thm:projA_gen} in the DR algorithm. The DR algorithm is applied to solving not only the above-mentioned example problems but also a challenging machine tool manipulator example problem.  These problems should furnish a class of test-bed examples for future studies.
    \item Selection of an algorithmic parameter plays an important role in the performance of the DR algorithm.  Through case studies, by means of the test-bed examples listed above, we explore and propose the ranges of values of this parameter with which the algorithms seem to converge in a smaller number of iterations.
\end{itemize}

We note that Corollary~\ref{cor:projA_PDI}, which provides an analytical projector expression in closed-form for the double integrator problem, was originally derived in~\cite[Proposition~1]{BausBuraKaya2019}.  Nevertheless, in this paper, we show that this expression can also be obtained using direct substitutions of the state transition matrix and the Jacobian into the general expression in Theorem~\ref{thm:projA_gen}.

For all the above-mentioned example problems we perform numerical experiments and compare the performance of the DR algorithm by also using the optimization modelling software AMPL paired with the interior point optimization software Ipopt.  In these experiments we observe that not only is the DR algorithm more efficient, i.e., it can find a solution in a much smaller amount of time, than the AMPL--Ipopt suite, but also that Ipopt sometimes fails in finding a solution at all.  We also compare the errors in the control and state variables separately.  These cases for different problems are tabulated altogether for an easier appreciation of the conclusions we set out.

The paper is organized as follows. Section~\ref{sec:oc} contains necessary background and preliminaries on minimum-energy control problems and optimality. In Section~\ref{sec:proj} we derive the projectors for a general minimum-energy control problem as well as some specific cases. Section \ref{sec:projAlg} presents the DR algorithm that we apply in Section \ref{sec:num}. Section \ref{sec:num} provides a numerical approach for obtaining the projector onto the affine set when it is not possible or convenient (due to length) to use an analytical expression. This section also contains numerical experiments comparing the performance of the DR algorithm with a direct discretization approach, as well as an exploration of (in some sense) best values of the parameter of the DR algorithm.  Section \ref{sec:con} contains concluding remarks and open problems.  In the appendix we provide the detailed proofs of the projectors onto the affine set for the harmonic oscillator problems.

\section{Minimum-energy Control Problem}
\label{sec:oc}
In this section we introduce the theoretical framework as well as the optimal control problem we study. We derive the necessary conditions of optimality for the problem via Pontryagin's maximum principle, which will be instrumental in the derivation of the projectors. We also split the constraints of the problem into two sets which facilitate the projection method we will study.

Before introducing the optimal control problem we will give some standard definitions. Unless otherwise stated all vectors are column vectors. Let ${\cal L}^2([t_0,t_f];\mathbb{R}^q)$ be the Hilbert space of Lebesgue measurable functions $z:[t_0,t_f]\rightarrow\mathbb{R}^q$, with finite ${\cal L}^2$ norm, namely,
\[{\cal L}^2([t_0,t_f];\mathbb{R}^q):=\left\{z:[t_0,t_f]\rightarrow\mathbb{R}^q\,\:|\,\:\|z\|_{{\cal L}^2} := \left(\int_{t_0}^{t_f}\|z(t)\|^2\,dt\right)^{1/2}<\infty\right\}\]
where $\|\cdot\|$ is the $\ell_2$ norm in $\dR^q$. Furthermore, ${\cal W}^{1,2}([t_0,t_f];\mathbb{R}^q)$ is the Sobolev space of absolutely continuous functions, namely
\[{\cal W}^{1,2}([t_0,t_f];\mathbb{R}^q):=\left\{z\in {\cal L}^2([t_0,t_f];\mathbb{R}^q)\,|\,\dot{z}:=dz/dt\in {\cal L}^2([t_0,t_f];\mathbb{R}^q)\right\},\]
endowed with the norm
\[\|z\|_{{\cal W}^{1,2}}:=(\|z\|_{{\cal L}^2}^2+\|\dot{z}\|_{{\cal L}^2}^2)^{1/2}.\]
With these definitions we define a general minimum-energy optimal control problem, which is an LQ control problem, as follows.
\[
\mbox{(P) }\left\{\begin{array}{rl}\label{eqn:OCP1}
\ds\min_{u} & \ \ \ds \frac{1}{2}\int_{t_0}^{t_f} \|u(t)\|^2 \,dt  \\[5mm] 
\mbox{subject to} & \ \ \dot{x}(t) = A(t)x(t)+B(t)u(t)\,,\ \ x(t_0) = x_0\,,\ \ x(t_f) = x_f\,, \\[2mm]
& \ \ u(t)\in U\subseteq\mathbb{R}^m\,,\ \ x(t)\in\mathbb{R}^n, \ \ \forall t\in[t_0,t_f].
\end{array}\right.
\]
The {\em state variable} $x\in {\cal W}^{1,2}([t_0,t_f];\mathbb{R}^n)$, with $x(t) := (x_1(t),\ldots,x_n(t))\in\dR^n$, and the {\em control variable} $u\in {\cal L}^2([t_0,t_f];\mathbb{R}^m)$, with $u(t) := (u_1(t),\ldots,u_m(t))\in\dR^m$. The set $U$ is a fixed closed subset of $\dR^m$. 
The time varying matrices $A:[t_0,t_f]\rightarrow\mathbb{R}^{n\times n}$ and $B:[t_0,t_f]\to\mathbb{R}^{n\times m}$ are continuous. The initial and terminal states are given as $x_0$ and $x_f$ respectively. Note that, for every $t\in[t_0,t_f]$, we can write
\[
B(t)u(t) = \sum_{i=1}^m b_i(t)\,u_i(t)\,,
\]
where $b_i(t)\in \mathbb{R}^n,\,i=1,\ldots,m$, is the $i$th column of $B(t)$. We note that, when (P) is feasible, it has a unique solution due to the strong convexity of the objective function.

We assume that (i) the dynamical system in~(P) is controllable, i.e., by choosing a suitable unconstrained control variable $u(\cdot)$, one can drive any initial state $x_0$ to any other terminal state $x_f$, (ii) Problem (P) is feasible, i.e., the constraint set of Problem~(P) is nonempty and (iii) Problem~(P) is normal, i.e., Pontryagin's maximum principle does not become degenerate and fail to provide information on optimality.

\subsection{Optimality conditions}\label{subsec:opCond}
In this section, we use Pontryagin's maximum principle to derive the necessary conditions of optimality for Problem (P).

Various forms of Pontryagin's maximum principle can be found, along with their proofs, in a number of reference books -- see, for example, \cite[Theorem 1]{pontry1962}, \cite[Chapter 7]{Hestenes66}, \cite[Theorem 6.4.1]{Vinter2000}, \cite[Theorem 6.37]{mord2006}, and \cite[Theorem 22.2]{clarke2013}. We will state Pontryagin's maximum principle using notation and settings from these references. We start by defining the {\em Hamiltonian function} $H:\mathbb{R}^n\times\mathbb{R}^m\times\mathbb{R}^n\times\mathbb{R}\times[t_0,t_f] \to \mathbb{R}$ for Problem~(P) as
\[
H(x(t),u(t),\lambda(t),\lambda_0,t):=\frac{\lambda_0}{2}\|u(t)\|^2 +\lambda(t)^T\left(A(t)x(t)+\sum_{i=1}^m b_i(t)\,u_i(t)\right),
\]
where the {\em adjoint variable} vector $\lambda:[t_0,t_f]\rightarrow\mathbb{R}^n$, with $\lambda(t):=(\lambda_1(t),\dots,\lambda_n(t))\in\mathbb{R}^n$ and $\lambda_0$ is a real constant. For brevity, we use the following short-hand notation,
\[
H[t] := H(x(t),u(t),\lambda(t),\lambda_0,t)\,.
\]
The adjoint variable vector is assumed to satisfy the condition (see e.g. \cite{Hestenes66})
\begin{equation}\label{eqn:adj}
    \dot{\lambda}(t) := -H_x[t] = -A(t)^T\lambda(t)
\end{equation}
for every $t\in[t_0,t_f]$, where $H_x:=\partial H/\partial x$. Suppose that the control set $U$ is a box in $\mathbb{R}^m$, i.e., $U = [-a_1,a_1]\times\cdots\times[-a_m,a_m]$, and that the pair $(x,u)\in {\cal W}^{1,2}([t_0,t_f];\mathbb{R}^n)\times {\cal L}^2([t_0,t_f];\mathbb{R}^m)$ is optimal for Problem~(P). Then Pontryagin's maximum principle asserts that there exist a real number $\lambda_0 \ge 0$ and a continuous adjoint variable vector $\lambda\in {\cal W}^{1,2}([t_0,t_f];\mathbb{R}^n)$ as defined in Equation~\eqref{eqn:adj}, such that $\lambda(t)\neq\mathbf{0}$ for all $t\in[t_0,t_f]$, and that, for all $t\in[t_0,t_f]$,
\begin{align}
    u_i(t) &= \argmin_{|v_i|\leq a_i} H(x(t),u_1(t),\ldots,v_i,\dots,u_m(t),\lambda(t),\lambda_0,t) \nonumber \\
    &= \argmin_{|v_i|\leq a_i} \frac{\lambda_0}{2}\,(u_1(t)^2 + \ldots + v_i^2 + \ldots + u_m(t)^2) \nonumber \\
    &\hspace*{20mm} +\lambda^T(t)\,(A(t)x(t)+ b_1(t)\,u_1(t) + \ldots + b_i(t)\,v_i + \ldots +b_m(t)\,u_m(t)) \nonumber \\
    &= \argmin_{|v_i|\leq a_i} \frac{\lambda_0}{2}\, v_i^2 +\lambda^T(t)\,b_i(t)\,v_i\,,
    \label{eqn:maxPrin}
\end{align}
for $i = 1,\ldots,m$. We ignored all terms that do not depend on $v_i$ to arrive at Equation~\eqref{eqn:maxPrin}. If $a_i=\infty$, $i = 1,\ldots,m$, i.e., if the control vector is unconstrained, then \eqref{eqn:maxPrin} becomes
\[H_{u_i}[t] = 0,\]
\begin{equation}  \label{ui_unconstr}
\lambda_0 u_i(t)+b_i(t)^T\lambda(t) = 0\,,
\end{equation}
$i = 1,\ldots,m$.  We assume that the problem is {\em normal}, i.e., $\lambda_0>0$, so we can take $\lambda_0=1$ without loss of generality.  Then \eqref{ui_unconstr} can be solved for $u_i(t)$ as
\begin{equation}  
u_i(t) = -b_i(t)^T\lambda(t)\,,
\end{equation}
for $i = 1,\ldots,m$; or using the input matrix $B(t)$,
\begin{equation} \label{eqn:u(t)}
u(t) = -B(t)^T\lambda(t)\,.
\end{equation}
With the box constraint on $u(t)$, one gets from~\eqref{eqn:maxPrin}
\begin{eqnarray}\label{eqn:u_gen}
u_i(t) = \left\{\begin{array}{rl}
   a_i, & \mbox{if } b_i^T(t)\lambda(t)\leq-a_i, \\[2mm]
  -b_i^T(t)\lambda(t), & \mbox{if } -a_i\leq b_i^T(t)\lambda(t)\leq a_i, \\[2mm]
  -a_i, & \mbox{if } b_i^T(t)\lambda(t)\geq a_i,
\end{array} \right.
\end{eqnarray}
for all $t\in[t_0,t_f]$, $i=1,\dots,m$.

Recall that the {\em state transition matrix} $\Phi_{A}(t,t_0)$ of $\dot{x}(t) = A(t)\,x(t)$, also referred to as the {\em resolvant matrix}, is the unique matrix such that $x(t) = \Phi_{A}(t,t_0)\,x(t_0)$---also see \cite{Rugh1995} for further details and the properties. Then from
 \cite{BorCol} the solution of the initial value problem $\dot{x}(t) = A(t)x(t)+B(t)u(t)$, $x(t_0) = x_0$, in Problem~(P) can simply be written as
\begin{equation} \label{x_soln}
x(t) = \Phi_A(t,t_0)\,x_0 + \int_{t_0}^t \Phi_A(t,\tau)\,B(\tau)\,u(\tau)\,d\tau\,.
\end{equation}
Similarly, Equation \eqref{eqn:adj} can be solved as $\lambda(t) = \Phi_{(-A^T)}(t,t_0)\lambda_0$, or by using the identity $\Phi_{(-A^T)}(t,t_0) = \Phi_A(t_0,t)^T$~\cite[Property~4.5]{Rugh1995},
\begin{equation}\label{eqn:lambda_gen}
    \lambda(t) = \Phi_A(t_0,t)^T\lambda_0\,.
\end{equation}
When $a_i$ is small enough so that the control constraint is active it is usually impossible to find an analytical solution for (P), hence the need for numerical methods.

\subsection{Constraint splitting}\label{subsec:cons}
We split the constraints into the two sets given below.
\begin{eqnarray} 
{\cal A} &:=& \big\{u\in {\cal L}^2([t_0,t_f];\dR^m)\ |\ \exists x\in {\cal W}^{1,2}([t_0,t_f];\dR^n)\mbox{ which solves } \nonumber \\
&&\ \ \dot{x}(t) = A(t)x(t)+B(t)u(t), \ \ x(t_0) = x_0, \ \ x(t_f) = x_f,\ \forall t\in[t_0,t_f]\big\}\,, \label{A_gen} \\[2mm]
{\cal B} &:=& \big\{u\in {\cal L}^2([t_0,t_f];\dR^m)\ | -a_i\le u_i(t)\le a_i\,,\ \forall t\in[t_0,t_f], i=1,\dots,m\big\}\,.\label{B_gen}
\end{eqnarray}
The set ${\cal A}$ is an {\em affine space} and contains all the feasible control functions from (P) where the control function is unconstrained. The set ${\cal B}$ is a {\em box} which contains all the control functions with components $u_i$ that are constrained by $-a_i$ and $a_i$ (where each $a_i$ is nonnegative). These two sets form the constraint sets for our two subproblems. The reason we split the original problem into two subproblems is because they are much simpler to solve individually so we can derive analytical expressions.

Recall that we assume the dynamical system in~\eqref{A_gen} is {\em controllable}, i.e., that there exists some control $u(\cdot)$ with which the system can be driven from any $x_0$ to any other $x_f$ (see~\cite{Rugh1995}), so that ${\cal A} \neq\emptyset$.  We also assume that ${\cal A} \cap {\cal B} \neq \emptyset$, namely that Problem~(P) is feasible.

\section{Projectors}\label{sec:proj}
In this section we give the projectors onto the sets ${\cal A}$ and ${\cal B}$ for a general problem~(P) followed by the projectors for some specific problems, namely the double integrator, pure harmonic oscillator and (under-, critically- and over-) damped harmonic oscillator.

\subsection{Projectors for general minimum-energy control}
Now that we have the constraint sets, we need to define the subproblems. First, recall that the \emph{projection} $P_C(x)$ of a point $x$ onto $C$ is characterized by $P_C(x)\in C$ and, $\forall y\in C$, $\langle y-P_C(x)|x-P_C(x)\rangle\leq0$ \cite[Theorem~3.16]{BauCombettes}. In our context, the projection onto ${\cal A}$ from a current iterate $u^-$ is the point $u$ which solves the following problem.
\[\mbox{(P1) }\left\{\begin{array}{rl}
\ds\min & \ \ \ds\frac{1}{2}\int_{t_0}^{t_f} \|u(t)-u^-(t)\|^2\,dt = \frac{1}{2} \|u - u^-\|_{{\cal L}^2}^2 \\[5mm] 
\mbox{subject to} & \ \ u\in{\cal A}.
\end{array} \right.\]
The projection onto ${\cal B}$ from a current iterate $u^-$ is the point $u$ which solves the following problem.
\[\mbox{(P2) }\left\{\begin{array}{rl}
\ds\min & \ \ \ds \frac{1}{2} \|u - u^-\|_{{\cal L}^2}^2 \\[5mm] 
\mbox{subject to} & \ \ u\in{\cal B}.
\end{array} \right.\]

First we provide a technical lemma.
\begin{lemma}\label{lem:A_tilde}
Given the $n\times n$ matrix $A(t)$, consider the $n^2\times n^2$ matrix $\widetilde{A}(t)$, defined as
\[
\widetilde{A}(t) := \begin{bmatrix}
A(t) & & \mathbf{0}\\
& \ddots & \\
\mathbf{0} & & A(t) 
\end{bmatrix},
\]
where $\mathbf{0}$ is a zero matrix of appropriate size, and the matrix $A(t)$ appears repeatedly ($n$ times) in diagonal blocks. The state transition matrix of $\widetilde{A}(t)$ is the $n^2\times n^2$ matrix defined as
\begin{equation}\label{eqn:phi_A_tilde}
\Phi_{\widetilde{A}}(t,t_0) := \begin{bmatrix}
\Phi_A(t,t_0) & & \mathbf{0}\\
& \ddots & \\
\mathbf{0} & & \Phi_A(t,t_0)
\end{bmatrix},
\end{equation}
where $\Phi_A(t,t_0)$ (the state transition matrix for $A(t)$), appears repeatedly ($n$ times) in diagonal blocks, where all other elements are zero. 
\end{lemma}
\begin{proof}
Suppose that $\Phi_{A}(t,t_0)$ is the state transition matrix of $\dot{y_i}(t) = A(t)\,y_i(t)$, $i = 1,\ldots,n$, where $y_i(t)\in\dR^n$.  Suppose that $y_i(t_0) = y_{i,0}$, $i = 1,\ldots,n$, are the initial conditions. By the definition preceding~\eqref{x_soln}, $y_i(t) = \Phi_{A}(t,t_0)\,y_{i,0}$ for $i = 1,\ldots,n$ is the unique solution. Then with $\widetilde{y}(t) := (y_1(t),\ldots,y_n(t))\in\dR^{n^2}$, we get $\dot{\widetilde{y}}(t) = \widetilde{A}(t)\,\widetilde{y}(t)$ and in turn $\Phi_{\widetilde{A}}(t,t_0)$ is as required by~\eqref{eqn:phi_A_tilde} in the lemma.
\end{proof}

Theorem~\ref{thm:projA_gen} further below furnishes an expression for the projector onto $\cal{A}$. Even though the proof of this theorem uses a classical shooting technique and broadly follows steps similar to those in~\cite[Proposition~1]{BausBuraKaya2019}, the case considered in Theorem~\ref{thm:projA_gen} is more general. To simplify presentation, we establish next a technical result involving the shooting concept.

\begin{lemma}\label{lem:shooting}
Fix $u^-\in {\cal L}^2([t_0,t_f];\mathbb{R}^m)$ and consider the following initial value problem
\begin{equation}\label{eqn:IVP-1}
\begin{array}{lcl}
\dot{z}(t,\lambda_0)&=& A(t){z}(t,\lambda_0) + B(t)\left[u^-(t) - B(t)^T\Phi_A(t_0,t)^T\lambda_0\right],\\
  z(t_0,\lambda_0)&= &x_0, 
\end{array}
\end{equation}
where the dependence of the solution $z(\cdot,\lambda_0)$ on the parameter $\lambda_0$ has been made explicit and, with a slight abuse of notation, $\dot{z}(t,\lambda_0):={\partial z(t,\lambda_0)}/{\partial t}$.  Then,
\begin{itemize}
    \item[(i)] $z(t_f,\cdot) $ is affine (i.e., $z(t_f,\lambda_0) $ is affine in the variable $\lambda_0$). In particular, its partial derivative w.r.t. $\lambda_0$ is a constant $n\times n$ matrix, which we denote as $\partial z(t_f,0)/\partial \lambda_0$.
    \item[(ii)] Let $z(\cdot,0)$ be the solution of \eqref{eqn:IVP-1} when $\lambda_0=0$.  Consider the solution $\overline{\lambda}$ of the linear system
\begin{equation}\label{eqn:Jacfi1}
\frac{\partial z(t_f,0)}{\partial \lambda_0} \overline{\lambda}=x_f-z(t_f,0).
\end{equation}
Then $\overline{\lambda}$ verifies 
\begin{equation}\label{eqn:TPBVP-2}
\begin{array}{lcl}
\dot{z}(t,\overline{\lambda})&=& A(t){z}(t,\overline{\lambda}) + B(t)\left[u^-(t) - B(t)^T\Phi_A(t_0,t)^T\overline{\lambda}\right],\\
  z(t_0,\overline{\lambda})&= &x_0, \\
 {z}(t_f,\overline{\lambda})&=&x_f\,.
\end{array}
\end{equation}
Equivalently,
\begin{equation}\label{eqn:A}
 u(t):=\left[u^-(t) - B(t)^T\Phi_A(t_0,t)^T\,\overline{\lambda}\right] \in \cal{A}.   
\end{equation}

\end{itemize}
\end{lemma}
\begin{proof}
(i) Using \eqref{eqn:adj} and \eqref{eqn:lambda_gen} we can rewrite the dynamics in \eqref{eqn:IVP-1} as the system
\[
\begin{array}{lcl}
\dot{z}(t,\lambda_0)&=& A(t){z}(t,\lambda_0) + B(t)\left[u^-(t) - B(t)^T\lambda(t)\right],\\
\dot{\lambda}(t)&=& -A^T(t)\lambda(t),
\end{array}
\]
which, using matrix notation, becomes
\begin{align}\label{eqn:zODE}
\begin{bmatrix}\dot{z}(t,\lambda_0) \\ \dot{\lambda}(t)\end{bmatrix} = \begin{bmatrix}A(t) & -B(t)B^T(t) \\ 0_{n\times n} & -A^T(t)\end{bmatrix}\begin{bmatrix}z(t,\lambda_0) \\ \lambda(t)\end{bmatrix} + \begin{bmatrix}B(t) \\ 0_{n\times m}\end{bmatrix}u^-(t)\,.
\end{align}
Let $g(t,\lambda_0):=[z(t,\lambda_0) \ \lambda(t)]^T$. To show (i) it is enough to prove that $g(t,\cdot)$ is affine in $\lambda_0$. Namely, we claim that
\begin{equation}\label{eqn:affine_g}
g(t,\alpha\lambda_1+(1-\alpha)\lambda_2)=\alpha g(t,\lambda_1)+(1-\alpha)g(t,\lambda_2),
\end{equation}
for all $\alpha\in\mathbb{R}$ and $\lambda_1,\lambda_2\in\mathbb{R}^n$. Let the first coefficient matrix on the right-hand side of \eqref{eqn:zODE} be denoted by $C(t)$ and the matrix multiplying $u^-(t)$ be denoted by $D(t)$. Solving \eqref{eqn:zODE} gives
\begin{align*}
g(t,\lambda_0) = \Phi_C(t,t_0)\begin{bmatrix}x_0 \\ \lambda_0\end{bmatrix} + \int_{t_0}^{t} \Phi_C(t,\tau)D(\tau)u^-(\tau)\, d\tau,
\end{align*}
where $\Phi_C(t,t_0)$ is the state transition matrix of $\dot{y}(t) = C(t)y(t)$, for all $t\in[t_0, t_f]$. Next we start off with the left-hand side of~\eqref{eqn:affine_g}, with the aim of getting the right-hand side after direct manipulations.
\begin{align*}
g(t,\alpha\lambda_1+(1-\alpha)\lambda_2) &= \Phi_C(t,t_0)\begin{bmatrix}x_0 \\ \alpha\lambda_1+(1-\alpha)\lambda_2\end{bmatrix}+\gamma(t),
\end{align*}
where $\gamma(t):=\int_{t_0}^{t}\Phi_A(t,\tau)D(\tau)u^-(\tau)\,d\tau$.  Continuing with further manipulations,
\begin{align*}
g(t,\alpha\lambda_1+(1-\alpha)\lambda_2) &= \Phi_C(t,t_0)\begin{bmatrix}\alpha x_0+(1-\alpha)x_0 \\ \alpha\lambda_1+(1-\alpha)\lambda_2\end{bmatrix}+\alpha\gamma(t)+(1-\alpha)\gamma(t) \\
&= \alpha\Phi_C(t,t_0)\begin{bmatrix}x_0 \\ \lambda_1\end{bmatrix}+(1-\alpha)\Phi_C(t,t_0)\begin{bmatrix}x_0 \\ \lambda_2\end{bmatrix}+\alpha\gamma(t)+(1-\alpha)\gamma(t) \\
&= \alpha\left(\Phi_C(t,t_0)\begin{bmatrix}x_0 \\ \lambda_1\end{bmatrix}+\gamma(t)\right)+(1-\alpha)\left(\Phi_C(t,t_0)\begin{bmatrix}x_0 \\ \lambda_2\end{bmatrix}+\gamma(t)\right) \\
&= \alpha g(t,\lambda_1)+(1-\alpha)g(t,\lambda_2),
\end{align*}
which verifies~\eqref{eqn:affine_g} and thus proves the affineness of $g(t,\cdot)$. Hence $z(t_f,\cdot)$ is affine too. The proof of (i) is complete. 

\noindent To prove (ii), we use the fact that $z(t_f,\cdot)$ is affine to write \[
z(t_f,\overline{\lambda})=z(t_f,0)+\frac{\partial z(t_f,0)}{\partial \lambda_0}\,\overline{\lambda} ,
\]
where $\overline{\lambda}$ is as in \eqref{eqn:Jacfi1} and $ {\partial z(t_f,0)}/{\partial \lambda_0}$ is the Jacobian of $z(t_f,\cdot)$ evaluated at $(t_f,0)$. Equation \eqref{eqn:Jacfi1} and the above equality yield
\[
z(t_f,\overline{\lambda})=z(t_f,0)+ (x_f- z(t_f,0))= x_f.
\]
Also note that, by definition, the function $z(\cdot,\overline{\lambda})$ must solve system \eqref{eqn:IVP-1} with $\overline{\lambda}$ in place of $\lambda_0$. Altogether, we have shown that $z(\cdot,\overline{\lambda})$ verifies \eqref{eqn:TPBVP-2}. The last statement of the lemma now follows from \eqref{eqn:TPBVP-2} and the definition of $\mathcal{A}$.
\end{proof}

The next definition points to the connection between Lemma \ref{lem:shooting} and the shooting method.

\begin{definition}\label{def:NMF}
Define the {\em near miss function} as
\begin{equation}\label{eqn:nearmiss}
\varphi(\lambda_0) := z(t_f,\lambda_0)-x_f\,,
\end{equation}
where $\lambda_0\in \mathbf{R}^n$ is arbitrary, and $z$ is as in Lemma \ref{lem:shooting}. Namely,
$z(t_f,\lambda_0)$ is a solution of system \eqref{eqn:IVP-1} evaluated at $(t_f,\lambda_0)$. The function $\varphi$ measures the discrepancy of a solution $z(\cdot,\lambda_0)$ of \eqref{eqn:IVP-1} at the end-point $t=t_f$. By Lemma \ref{lem:shooting}(i) and \eqref{eqn:nearmiss}, $\varphi$ is affine and so its Jacobian $ J_{\varphi}(\lambda_0)$ is a constant matrix such that
\begin{equation}\label{eqn:Jacfi}
    J_{\varphi}(\lambda_0)=J_{\varphi}(0)=\frac{\partial z(t_f,0)}{\partial \lambda_0}.
\end{equation}
In particular, for every $\lambda_0$ we can write
\[
\varphi(\lambda_0)=\varphi(0)+ J_{\varphi}(0) \lambda_0.
\]   
\end{definition}

We are now ready to establish our formula for the projection onto $\mathcal{A}$.

\begin{theorem}  \label{thm:projA_gen}
The projection $P_{\cal{A}}$ of $u^-\in {\cal L}^2([t_0,t_f];\mathbb{R}^m)$ onto the constraint set $\cal{A}$, as the solution of Problem~{\em (P1)}, is given by
\begin{equation}\label{eqn:projA}
P_{\cal{A}}(u^-)(t) = u^-(t) - B(t)^T\,\Phi_A(t_0,t)^T\,\overline{\lambda}\,,
\end{equation}
for all $t\in[t_0,t_f]$, where $\overline{\lambda}$ solves
\begin{equation}\label{eqn:J1}
   \frac{\partial y(t_f)}{\partial \lambda_0}\,\overline{\lambda}= -(y(t_f)-x_f)\,,
\end{equation}
where
\[
y(t_f) := \Phi_A(t_f,t_0)\,x_0 + \int_{t_0}^{t_f} \Phi_A(t_f,\tau)\,B(\tau)\,u^-(\tau)\,d\tau\,.
\]
Moreover,
\begin{equation}\label{eqn:NMF}
  \frac{\partial y(t_f)}{\partial \lambda_0}= J_{\varphi}(0),
    \end{equation}
    where $\varphi$ and $J_{\varphi}(0)$ are as in Definition \ref{def:NMF}.
\end{theorem}
\begin{proof}
The Hamiltonian for Problem (P1) is
\begin{equation*}
H[t] := \frac{1}{2}\|u(t)-u^-(t)\|^2+\lambda^T(t)(A(t)x(t)+B(t)u(t)).
\end{equation*}
From Pontryagin's maximum principle, $H_u[t] = 0$ and so
\begin{align}\label{eqn:gen_u}
u(t) = u^-(t)-B(t)^T\lambda(t)\,,  
\end{align}
for all $t\in[t_0,t_f]$.  To fully solve this equation, we need to find $\lambda(\cdot)$ such that the function $u$ on the left hand-side of \eqref{eqn:gen_u} belongs to ${\cal A}$. Equation \eqref{eqn:lambda_gen} gives $\lambda(\cdot)$ as a function of an initial condition $\lambda_0$. The aim is therefore to determine an initial condition such that the corresponding $\lambda(\cdot)$ produces a function $u$ in ${\cal A}$. To avoid confusion, we call $\overline{\lambda}$ this desired initial condition. We now use the last statement of Lemma \ref{lem:shooting}, which states that $\overline{\lambda}$ as in \eqref{eqn:Jacfi1} ensures that $u$ as in \eqref{eqn:A} is in $\mathcal{A}$. Altogether, this choice of $\overline{\lambda}$ verifies
\begin{align}
u(t) = u^-(t)-B(t)^T\lambda(t)=u^-(t)-B(t)^T\Phi_A(t_0,t)^T \,\overline{\lambda}\in \mathcal{A}\,,  \label{eqn:gen_u2}
\end{align}
where we used \eqref{eqn:A} in the inclusion and in the second equality we used \eqref{eqn:lambda_gen}   with $\overline{\lambda}$ in place of $\lambda_0$.   Hence, $u$ is the desired projection onto $\mathcal{A}$ for this choice of $\overline{\lambda}$.

To establish the rest of the theorem, call $y(\cdot):=z(\cdot,0)$, where $z(\cdot,0)$ is a solution of the IVP \eqref{eqn:IVP-1} for $\lambda_0:=0$. Therefore, we have
\begin{equation}  \label{eqn:lineq_xf}
y(t_f) = \Phi_A(t_f,t_0)\,x_0 + \int_{t_0}^{t_f}\,\Phi_A(t_f,\tau)\,B(\tau)\,u^-(\tau)\,d\tau\,.
\end{equation}
Since $y(\cdot) = z(\cdot,0)$ we have that
\[
\frac{\partial z(t_f,0)}{\partial \lambda_0}=\frac{\partial y(t_f)}{\partial \lambda_0},
\]
so condition \eqref{eqn:Jacfi1} in the lemma becomes \eqref{eqn:J1}. The last statement of the theorem follows from the above equality and \eqref{eqn:Jacfi}.
\end{proof}


\begin{corollary}\label{cor: A closed}
Let $X$ be a Hilbert space and let $C\subset X$ be a nonempty convex set. Assume that for every $u\in X$ there exists the projection $P_C(u)$ of $u$ onto $C$. Then the set $C$ must be closed.    
\end{corollary}
\begin{proof}
Denote by ${\rm cl\, }C$ the closure of $C$ and take any $z\in {\rm cl\, }C$. 
By assumption on $C$, the projection $P_C(z)$ of $z$ onto $C$ exists. The definitions imply that
\[
\|z-P_C(z)\|=d(z,C)=d(z,{\rm cl\, }C)=0.
\]
Hence, $z=P_C(z)\in C$. So ${\rm cl \,}C\subset C$ and therefore $C$ is closed.
\end{proof}

\begin{remark} \rm
One can find the elements of ${\partial y(t_f)}/{\partial \lambda_0}=J_{\varphi}(0)$ by solving the variational equations 
\begin{equation}  \label{eqn:z}
 \frac{\partial z(t,\lambda_0)}{\partial t} = A(t)z(t,\lambda_0) + B(t)\left[u^-(t) - B(t)^T\Phi_A(t_0,t)^T\lambda_0\right],\ \ z(t_0,\lambda_0)=x_0\,.
 \end{equation}
with respect to $\lambda_0$, i.e., by solving the following equations in $(\partial z / \partial \lambda_{0,i})(t,\lambda_0) \in \dR^n$, for $i=1,\dots,n$.
\[
\frac{\partial}{\partial t}\left(\frac{\partial z}{\partial \lambda_{0,i}}\right)(t,\lambda_0) =  A(t)\frac{\partial z}{\partial \lambda_{0,i}}(t,\lambda_0) - B(t)B(t)^T\Phi_A(t_0,t)^T\,e_i
\]
where $e_i\in\mathbb{R}^n$ are the canonical basis vectors, i.e, with 1 in the $i$th coordinate and zero elsewhere. Let $\widetilde{y}(t),\dot{\widetilde{y}}(t)\in\mathbb{R}^{n^2}$ where
\begin{equation*}
\widetilde{y} := \begin{bmatrix}
\partial z/\partial \lambda_{0,1} \\
\vdots \\
\partial z/\partial \lambda_{0,n}
\end{bmatrix} = \begin{bmatrix}
y_1 \\ \vdots \\ y_n
\end{bmatrix} \ \
\textrm{ and } \ \
\dot{\widetilde{y}} := \begin{bmatrix}
\frac{\partial}{\partial t}\left(\partial z/\partial \lambda_{0,1}\right) \\
\vdots \\
\frac{\partial}{\partial t}\left(\partial z/\partial \lambda_{0,n}\right)
\end{bmatrix}= \begin{bmatrix}
\dot{y}_1 \\ \vdots \\ \dot{y}_n
\end{bmatrix},
\end{equation*}
then $\dot{\widetilde{y}}=\widetilde{A}(t)\widetilde{y}+\widetilde{B}(t),\ \widetilde{y}(t_0)=0$ where
\begin{equation*}
\widetilde{A}(t) = \begin{bmatrix}
A(t) & & \mathbf{0}\\
& \ddots & \\
\mathbf{0} & & A(t) 
\end{bmatrix}
\in\mathbb{R}^{n^2\times n^2} \ \
\textrm{ and } \ \
\widetilde{B}(t) = -B(t)B(t)^T\Phi_A(t_0,t)^T\begin{bmatrix}
e_1 \\
\vdots \\
e_n
\end{bmatrix}
\in\mathbb{R}^{n^2}.
\end{equation*}
Using Lemma \ref{lem:A_tilde} along with knowledge of differential equations
\begin{equation}\label{eqn:y}
\widetilde{y}(t) = \int_{t_0}^{t}\Phi_{\widetilde{A}}(t,\tau)\widetilde{B}(\tau)\,d\tau,
\end{equation}
where ${\Phi}_{\widetilde{A}}(t,t_0)$ is the transition matrix of $\dot{\widetilde{y}}=\widetilde{A}(t)\widetilde{y}$. So evaluating the above integral and substituting $t=t_f$ gives the components of ${\partial y(t_f)}/{\partial \lambda_0}$.
\proofbox
\end{remark}


\begin{theorem}[Projection onto \boldmath{${\cal B}$}]\label{thm:projB}
The projection $P_{{\cal B}}$ of $u^-\in {\cal L}^2([t_0,t_f];\dR)$ onto the constraint set ${\cal B}$, as the solution of Problem~{\em(P2)},  is given by
\begin{eqnarray}\label{eqn:projB}
[P_{{\cal B}}(u^-)(t)]_i = \left\{\begin{array}{rl}
   a_i, & \mbox{if } u^-_i(t) \geq a_i, \\[2mm]
  u^-_i(t), & \mbox{if } -a_i\leq u^-_i(t) \leq a_i, \\[2mm]
  -a_i, & \mbox{if } u^-_i(t) \leq -a_i,
\end{array} \right.
\end{eqnarray}
for all $t\in[t_0,t_f],\ i=1,\dots,m$. \end{theorem}
\begin{proof}
Simply use separability of Problem~(P2) in $u_i$, $i = 1,\ldots,m$.
\end{proof}

\subsection{Projectors for special cases}
\label{subsec:proj_2var}

In this subsection we consider problems with two state variables ($n = 2$) and one control variable ($m=1$).  In particular we consider problems involving the double integrator and the pure and damped harmonic oscillators, for which the general system and control matrices in set ${\cal A}$ in~\eqref{A_gen} become
\begin{equation}\label{eqn:matrices}
A(t) = A = \begin{bmatrix}
0 & 1 \\
-\omega_0^2 & -2\zeta\omega_0
\end{bmatrix}\quad\mbox{and}\quad
B(t) = b = \begin{bmatrix}
0 \\
1
\end{bmatrix},
\end{equation}
where $\omega_0$ is the natural frequency and $\zeta$ is the damping ratio.  Note that $\zeta = 0$ is the case of pure (undamped) harmonic oscillator, and $0<\zeta<1$ under-damped, $\zeta=1$ critically-damped, and $\zeta>1$ over-damped harmonic oscillator.  The general forms of the constraint sets can be found in \eqref{A_gen}--\eqref{B_gen} but for this specialization we define
\begin{eqnarray} 
{\cal A}_{\omega_0,\zeta} &:=& \big\{u\in {\cal L}^2([0,t_f];\dR)\ |\ \exists x\in {\cal W}^{1,2}([0,t_f];\dR^2)\mbox{ which solves } \nonumber \\
&&\ \ \dot{x}_1(t) = x_2(t)\,,\ x_1(0) = s_0\,,\ x_1(t_f) = s_f\,, \nonumber \\[1mm]
&&\ \ \dot{x}_2(t) = -\omega_0^2 x_1(t) -2\zeta\omega_0 x_2(t) + u(t)\,,\ \ \,x_2(0) = v_0\,,\ x_2(t_f) = v_f\,,\nonumber\\[1mm] 
&&\ \ \forall t\in[0,t_f]\big\}\,, \label{A} \\[2mm]
{\cal B} &:=& \big\{u\in {\cal L}^2([0,t_f];\dR)\ | -a\le u(t)\le a\,,\ \forall t\in[0,t_f]\big\}\,.\label{B}
\end{eqnarray}
To maintain the flow of this paper we move the proofs from this subsection to \ref{sec:proofs} as these are rather lengthy and the proof techniques follow a similar pattern. In the lemmas we complete some of the technical steps by deriving expressions for the state transition matrices and Jacobians required in Theorem \ref{thm:projA_gen}. Then the corollaries follow by direct substitution into the expression \eqref{eqn:projA} in Theorem~\ref{thm:projA_gen}. Since we find analytical expressions in each of the 
lemmas we express the inverse of the Jacobian and use it directly in the expression in Theorem~\ref{thm:projA_gen}.

In the case where the inverse of the Jacobian is analytical and not lengthy we express $\lambda_0$ as
\[
\lambda_0 = -[J_\varphi(0)]^{-1}(x(t_f)-x_f)
\]
to have a more closed form expression for the projector:
\begin{multline}\label{eqn:projA_spec}
P_{\cal{A}}(u^-)(t) = u^-(t) + B(t)^T\Phi_A(t_0,t)^T[J_\varphi(0)]^{-1}\bigg(\Phi_A(t_f,t_0)\,x_0 \\ + \int_{t_0}^{t_f}\Phi_A(t_f,\tau)B(\tau)u^-(\tau)d\tau-x_f\bigg).
\end{multline}
In the cases of the double integrator as well as the under-, critically- and over-damped harmonic oscillators the inverse of the Jacobian is simple enough, so we will use \eqref{eqn:projA_spec}.

\subsubsection{Double integrator}\label{subsec:projPDI}
The dynamics of the double integrator are given by $\ddot{y}(t) = f(t)$, where $f(t)$ stands for forcing, which typically models the motion of a point mass (or analogously, an electric circuit or a fluid system with capacitance)---see pertaining references in \cite{BausBuraKaya2019}, where $y(t)$ is the position and $\dot{y}(t)$ the velocity at time $t$. With $x_1 := y$ and $x_2 := \dot{y}$, one gets the state equations $\dot{x}_1 = x_2$ and $\dot{x}_1 = u$; in other words, $\omega_0=0$ and $\zeta=0$, resulting in the constraint set ${\cal A}_{0,0}$.  We note from~\eqref{eqn:matrices} that
\[
A = \begin{bmatrix}
0 &\ 1 \\
0 &\ 0
\end{bmatrix}. \ \
\]
The minimum-energy problem that we consider corresponds, for example to the practical problem of engineering where one would like to minimize the average magnitude of the force, or the problem of designing cubic (variational) curves.

In what follows we present the projections onto ${\cal A}_{0,0}$ and $\cal B$ in the Corollaries \ref{cor:projA_PDI} and \ref{cor:projB} below. These two results and their proofs can be found in \cite{BausBuraKaya2019}.

Recall the definition of the state transition matrix $\Phi_A(t,t_0)$ via \eqref{x_soln} and the definition of the Jacobian $J_\varphi(0)$ in \eqref{eqn:Jacfi}.  The following lemma evaluates $\Phi_A(t,0)$ and $[J_\varphi(0)]^{-1}$ for the double integrator, which are utilized in the proof of Corollary~\ref{cor:projA_PDI}.
\begin{lemma}[Computation of $\Phi_A$ and $J_\varphi$ for $\omega_0=0,\,\zeta=0$] \label{lem:PDE}
One has that
\begin{equation}  \label{PDI_Phi}
    \Phi_A(t,0) = e^{At} = \begin{bmatrix}
1 & t \\
0 & 1
\end{bmatrix}, \quad
[J_\varphi(0)]^{-1} = \begin{bmatrix}
-12 & 6 \\ -6 & 2
\end{bmatrix}.
\end{equation}
\end{lemma}
\begin{proof}
See \hyperlink{lem:PDE_proof}{proof} in \ref{sec:proofs}.
\end{proof}

The following result is a corollary of Theorem \ref{thm:projA_gen}.  As mentioned already, this result along with its proof can be found in~\cite{BausBuraKaya2019}, but we present here a new proof which directly substitutes the expressions in Lemma~\ref{lem:PDE} into Theorem~\ref{thm:projA_gen}.
\begin{corollary}
[Projection onto ${\cal A}_{0,0}$~\cite{BausBuraKaya2019}]  \label{cor:projA_PDI}
The projection $P_{{\cal A}_{0,0}}$ of $u^-\in {\cal L}^2([0,1];\dR)$ onto the constraint set ${\cal A}_{0,0}$, as the solution of Problem~{\em(P1)} with $\omega_0=0$ and $\zeta=0$, is given by
\begin{equation}\label{eqn:u_PDI}
P_{{\cal A}_{0,0}}(u^-)(t) = u^-(t) + c_1\,t + c_2\,,
\end{equation}
for all $t\in[0,1]$, where
{\begin{eqnarray}
&& c_1 := 12\left(s_0+v_0-s_f+\int_0^1 (1-\tau)u^-(\tau)d\tau\right)-6\left(v_0-v_f+\int_0^1 u^-(\tau)d\tau\right), \\[1mm]
&& c_2 := -6\left(s_0+v_0-s_f+\int_0^1 (1-\tau)u^-(\tau)d\tau\right)+2\left(v_0-v_f+\int_0^1 u^-(\tau)d\tau\right).
\end{eqnarray}}
\end{corollary}
\begin{proof}
See \hyperlink{cor:PDE_proof}{proof} in \ref{sec:proofs}.
\end{proof}

The following result is a direct consequence of  Theorem~\ref{thm:projB} for all cases of the harmonic oscillator. 

\begin{corollary}[Projection onto \boldmath{${\cal B}$}~\cite{BausBuraKaya2019}]\label{cor:projB}
The projection $P_{{\cal B}}$ of $u^-\in {\cal L}^2([0,t_f];\dR)$ onto the constraint set ${\cal B}$, as the solution of Problem~{\em(P2)},  is given by
\begin{equation}  \label{u_proj_B}
P_{{\cal B}}(u^-)(t) = \left\{\begin{array}{rl}
a\,, &\ \ \mbox{if\ \ } u^-(t)\ge a\,, \\[1mm]
u^-(t)\,, &\ \ \mbox{if\ \ } -a\le u^-(t)\le a\,, \\[1mm]
-a\,, &\ \ \mbox{if\ \ } u^-(t)\le -a\,, \\[1mm]
\end{array} \right.
\end{equation}
for all $t\in[0,t_f]$. 
\end{corollary}

\subsubsection{Pure harmonic oscillator}\label{subsec:projPHO}
When a spring is added to the point mass, or an inductor to the electric circuit with a capacitor, one gets the {\em pure} (or undamped) harmonic oscillator, as without forcing, once excited the state variables will exhibit sustained oscillations (or sinusoids) at frequency $\omega_0$.  We extend Corollary~\ref{cor:projA_PDI} to the general case of projecting onto ${\cal A}_{\omega_0,0}$. In this case, from \eqref{eqn:matrices} one has
\begin{equation}  \label{A_PHO}
A = \begin{bmatrix}
0 & 1 \\ -\omega_0^2 & 0
\end{bmatrix}.
\end{equation}
The following lemma provides major ingredients for the projector in Theorem~\ref{thm:projA_gen}.
\begin{lemma}[Computation of $\Phi_A$ and $J_\varphi$ for $\omega_0>0,\,\zeta=0$]\label{lem:PHO}
One has that
\begin{equation} \label{eqn:PHO_Phi}
\Phi_A(t,0) = e^{At} = \begin{bmatrix}
\cos(\omega_0t) & \ds\frac{\sin(\omega_0t)}{\omega_0} \\ -\omega_0\sin(\omega_0t) & \cos(\omega_0t)
\end{bmatrix}, \quad
[J_\varphi(0)]^{-1} = \begin{bmatrix}
-\omega_0^2/\pi & 0 \\ 0 & -1/\pi
\end{bmatrix}.
\end{equation}
\end{lemma}
\begin{proof}
See \hyperlink{lem:PHO1_proof}{proof} in \ref{sec:proofs}.
\end{proof}

The following corollary is a direct consequence of Theorem~\ref{thm:projA_gen}.
\begin{corollary}
[Projection onto ${\cal A}_{\omega_0,0}$]  \label{cor:projA_PHO}
The projection $P_{{\cal A}_{\omega_0,0}}$ of $u^-\in {\cal L}^2([0,2\pi];\dR)$ onto the constraint set ${\cal A}_{\omega_0,0}$, as the solution of Problem~{\em(P1)}, $\zeta=0$, is given by
\begin{equation}  \label{eqn:u_PHO}
P_{{\cal A}_{\omega_0,0}}(u^-)(t) = u^-(t)+c_1\sin(\omega_0t) - c_2\cos(\omega_0t),
\end{equation}
where
\[
\begin{array}{l}
  c_1:=\displaystyle \frac{\omega_0}{\pi}\bigg(s_0-s_f-\frac{1}{\omega_0}\int_0^{2\pi}\sin(\omega_0\tau)u^-(\tau)\,d\tau\bigg)   ,  \\
     \\
    c_2:=\displaystyle \frac{1}{\pi}\bigg(v_0-v_f+\int_0^{2\pi}\cos(\omega_0\tau)u^-(\tau)\,d\tau\bigg).
\end{array}
\]

\end{corollary}
\begin{proof}
See \hyperlink{cor:PHO1_proof}{proof} in \ref{sec:proofs}.
\end{proof}

\subsubsection{Damped harmonic oscillator}\label{subsec:projPDHO}
If a damper (which is an element that dissipates energy) is added to the mass-spring system (or analogously a resistor added to a capacitor--inductor electrical circuit) one gets what is referred to as a {\em damped harmonic oscillator}. There are three cases to consider for a damped system, namely the {\em critically-} ($\zeta = 1$), {\em over-} ($\zeta > 1$) and {\em under-damped} ($0 < \zeta < 1$) cases. We provide the projectors for each case.

Corollary~\ref{cor:projA_PDHO1} below presents the projector onto the set ${\cal A}_{\omega_0,1}$ for the critically-damped case.  From \eqref{eqn:matrices}, the system matrix $A$ for this case is
\[A=\begin{bmatrix}
0 & 1 \\ -\omega_0^2 & -2\omega_0
\end{bmatrix}.\]
\begin{lemma}[Computation of $\Phi_A$ and $J_{\varphi}$ for $\omega_0>0,\,\zeta=1$]\label{lem:PDHO_crit}
One has that
\begin{align}\label{eqn:PDHO1_Phi}
\Phi_A(t,0) = e^{At} = e^{-\omega_0t}\begin{bmatrix}
\omega_0t+1 & t \\ -t\omega_0^2 & -\omega_0t+1
\end{bmatrix}
\end{align}
and
\begin{equation}\label{eqn:Jphi_PDHO}
[J_\varphi(0)]^{-1} = \dfrac{1}{y_{11}(2\pi)y_{22}(2\pi)-y_{12}(2\pi)y_{21}(2\pi)}\begin{bmatrix}
y_{22}(2\pi) & -y_{21}(2\pi) \\ -y_{12}(2\pi) & y_{11}(2\pi)
\end{bmatrix}
\end{equation}
where $y_{ij}(2\pi)$, $j = 1,2$, are the components of the vectors $y_i(2\pi)$, $i = 1,2$, given below.
\begin{equation}\label{eqn:y_PDHO1}
y(2\pi) = \left[\begin{array}{c}
y_1(2\pi) \\[1mm] \hdashline[3pt/3pt] \\[-4mm]
y_2(2\pi)
\end{array}\right] = 
\left[\begin{array}{c}
\dfrac{e^{-2\pi\omega_0}(2\pi\omega_0-e^{4\pi\omega_0}+2\pi\omega_0e^{4\pi\omega_0}+1)}{4\omega_0^3} \\[4mm] \dfrac{\pi\sinh(2\pi\omega_0)}{\omega_0} \\[4mm] \hdashline[3pt/3pt] \\[-2mm]
-\dfrac{\pi\sinh(2\pi\omega_0)}{\omega_0} \\[4mm] -\dfrac{e^{-2\pi\omega_0}(2\pi\omega_0+e^{4\pi\omega_0}+2\pi\omega_0e^{4\pi\omega_0}-1)}{4\omega_0}
\end{array}\right].
\end{equation}
\end{lemma}
\begin{proof}
See \hyperlink{lem:PDHO1_proof}{proof} in \ref{sec:proofs}.
\end{proof}

\begin{corollary}[Projection onto ${\cal A}_{\omega_0,1}$]  \label{cor:projA_PDHO1}
The projection $P_{{\cal A}_{\omega_0,\zeta}}$ of $u^-\in {\cal L}^2([0,2\pi];\dR)$ onto the constraint set ${\cal A}_{\omega_0,1}$, as the solution of Problem~{\em(P1)} with $\zeta=1$, is given by
\begin{align}\label{eqn:u_PDHO1}
P_{{\cal A}_{\omega_0,1}}(u^-)(t) = u^-(t)&+\frac{e^{\omega_0(t-2\pi)}}{y_{11}y_{22}-y_{12}y_{21}}\left(-((y_{22}+y_{12}\omega_0)t+y_{12})\left(x_1(2\pi)-\frac{s_f}{e^{-2\pi\omega_0}}\right)\right.\nonumber\\ &+\left.((y_{21}+y_{11}\omega_0)t+y_{11})\left(x_2(2\pi)-\frac{v_f}{e^{-2\pi\omega_0}}\right)\right),
\end{align}
where $y_{ij}$ are the components of $y(2\pi)$ given in~\eqref{eqn:y_PDHO1}.
\end{corollary}
\begin{proof}
See \hyperlink{cor:PDHO1_proof}{proof} in \ref{sec:proofs}.
\end{proof}

In Corollary~\ref{cor:projA_PDHO2} we consider the derivation of the projection onto the set ${\cal A}_{\omega_0,\zeta}$ from \eqref{A} where $\zeta>1$.
\begin{lemma}[Computation of $\Phi_A$ and $J_\varphi$ for $\omega_0>0,\zeta>1$]\label{lem:PDHO_2}
One has that 
\begin{equation}  \label{eqn:PDHO2_Phi}
\Phi_A(t,0) = e^{At} = \frac{e^{-\alpha t}}{\beta}\begin{bmatrix}
\omega_0\sinh(\beta t+\eta) & \sinh(\beta t) \\ -\omega_0^2\sinh(\beta t) & \omega_0\sinh(-\beta t+\eta)
\end{bmatrix}
\end{equation}
with $\alpha=\omega_0\zeta$, \ $\beta=\omega_0\sqrt{\zeta^2-1}, \ \eta = \frac{1}{2}\ln\left|\frac{\beta+\alpha}{\beta-\alpha}\right|$.  Then we express the inverse of the Jacobian as in \eqref{eqn:Jphi_PDHO} where $y_{ij}(2\pi)$, $j = 1,2$, are the components of the vectors $y_i(2\pi)$, $i = 1,2$, given below.
\begin{equation}\label{eqn:y_PDHO2}
y(2\pi) = \left[\begin{array}{c}
y_1(2\pi) \\[1mm] \hdashline[3pt/3pt] \\[-4mm]
y_2(2\pi)
\end{array}\right] = 
\left[\begin{array}{c}
\dfrac{e^{-2\pi\alpha}}{4\omega_0^2}\left(\dfrac{(1-e^{4\pi\alpha})\cosh(2\pi\beta)}{\alpha}+\dfrac{(1+e^{4\pi\alpha})\sinh(2\pi\beta)}{\beta}\right) \\[4mm]
\dfrac{\sinh(2\pi\alpha)\sinh(2\pi\beta)}{2\alpha\beta} \\[4mm] \hdashline[3pt/3pt] \\[-2mm]
-\dfrac{\sinh(2\pi\alpha)\sinh(2\pi\beta)}{2\alpha\beta} \\[4mm]
\dfrac{e^{-2\pi\alpha}}{4}\left(\dfrac{(1-e^{4\pi\alpha})\cosh(2\pi\beta)}{\alpha}+\dfrac{(1+e^{4\pi\alpha})\sinh(2\pi\beta)}{\beta}\right)
\end{array}\right].
\end{equation}
\end{lemma}
\begin{proof}
See \hyperlink{lem:PDHO2_proof}{proof} in \ref{sec:proofs}.
\end{proof}

\begin{corollary}[Projection onto ${\cal A}_{\omega_0,\zeta}$ with $\zeta>1$]\label{cor:projA_PDHO2}
The projection $P_{{\cal A}_{\omega_0,\zeta}}$ of $u^-\in {\cal L}^2([0,2\pi];\dR)$ onto the constraint set ${\cal A}_{\omega_0,\zeta}$, as the solution of Problem~{\em(P1)} where $\zeta>1$, is given by
\begin{align}  \label{eqn:u_PDHO2}
P_{{\cal A}_{\omega_0,\zeta}}(u^-)&(t) = u^-(t)+\dfrac{e^{\alpha (t-2\pi)}}{\beta^2(y_{11}(2\pi)y_{22}(2\pi)-y_{12}(2\pi)y_{21}(2\pi))}\bigg(-(y_{22}(2\pi)\sinh(\beta t)+y_{12}(2\pi)\nonumber\\ & \ \ \ \ \ \ \ \times\omega_0\sinh(\beta t+\eta))\left(x_1(2\pi)-\dfrac{\beta s_f}{e^{-2\pi\alpha}}\right)\nonumber\\ &\ \ \ \ \ \ \ +(y_{21}(2\pi)\sinh(\beta t)+y_{11}(2\pi)\omega_0\sinh(\beta t+\eta))\left(x_2(2\pi)-\dfrac{\beta v_f}{e^{-2\pi\alpha}}\right)\!\bigg)
\end{align}
where $\alpha=\omega_0\zeta, \ \beta=\omega_0\sqrt{\zeta^2-1}$ and $y_{ij}$ are the components of $y(2\pi)$ given in \eqref{eqn:y_PDHO2}.
\end{corollary}
\begin{proof}
See \hyperlink{cor:PDHO2_proof}{proof} in \ref{sec:proofs}.
\end{proof}

In Corollary~\ref{cor:projA_PDHO3} we consider the final case for the damped harmonic oscillator, which is the projection onto the set ${\cal A}_{\omega_0,\zeta}$ from \eqref{A} where $0<\zeta<1$.

\begin{lemma}[Computation of $\Phi_A$ and $J_\varphi$ for $\omega_0>0,0<\zeta<1$]\label{lem:PDHO_3}
One has that
\begin{equation}  \label{eqn:PDHO3_Phi}
\Phi_A(t,0) = e^{At} = \dfrac{e^{-\alpha t}}{\tb}\begin{bmatrix}
\omega_0\cos(\tb t+\gamma) & \sin(\tb t) \\ -\omega_0^2\sin(\tb t) & \omega_0\cos(\tb t-\gamma)
\end{bmatrix},
\end{equation}
where $\alpha=\omega_0\zeta, \ \tb=\omega_0\sqrt{1-\zeta^2}, \ \gamma = \tan^{-1}(-\frac{\alpha}{\tb})$. Then we express the inverse of the Jacobian as in \eqref{eqn:Jphi_PDHO} where $y_{ij}(2\pi)$, $j = 1,2$, are the components of the vectors $y_i(2\pi)$, $i = 1,2$, given below.
\begin{equation}\label{eqn:y_PDHO3}
y(2\pi) = \left[\begin{array}{c}
y_1(2\pi) \\[1mm] \hdashline[3pt/3pt] \\[-4mm]
y_2(2\pi)
\end{array}\right] = \left[\begin{array}{c}
\dfrac{e^{-2\pi\alpha}}{4\omega_0^2}\left(\dfrac{\cos(2\pi\tb)(1-e^{4\pi\alpha})}{\alpha}+\dfrac{\sin(2\pi\tb)(1+e^{4\pi\alpha})}{\tb}\right) \\[4mm]
\dfrac{\sinh(2\pi\alpha)\sin(2\pi\tb)}{2\alpha\tb} \\[4mm] \hdashline[3pt/3pt] \\[-2mm]
-\dfrac{\sinh(2\pi\alpha)\sin(2\pi\tb)}{2\alpha\tb} \\[4mm]
\dfrac{e^{-2\pi\alpha}}{4}\left(\dfrac{\cos(2\pi\tb)(1-e^{4\pi\alpha})}{\alpha}-\dfrac{\sin(2\pi\tb)(1+e^{4\pi\alpha})}{\tb}\right)
\end{array}\right].
\end{equation}
\end{lemma}
\begin{proof}
See \hyperlink{lem:PDHO3_proof}{proof} in \ref{sec:proofs}.
\end{proof}

\begin{corollary}[Projection onto ${\cal A}_{\omega_0,\zeta}$ with $0<\zeta<1$]\label{cor:projA_PDHO3}
The projection $P_{{\cal A}_{\omega_0,\zeta}}$ of $u^-\in {\cal L}^2([0,2\pi];\dR)$ onto the constraint set ${\cal A}_{\omega_0,\zeta}$, as the solution of Problem~{\em(P1)} where $0<\zeta<1$, is given by
\begin{align*}
P_{{\cal A}_{\omega_0,\zeta}} &= u^-(t)+\dfrac{e^{\alpha(t-2\pi)}}{\tb^2(y_{11}(2\pi)y_{22}(2\pi)-y_{12}(2\pi)y_{21}(2\pi))}\\&\Bigg(-\left(y_{22}(2\pi)\sin(\tb t)+y_{12}(2\pi)\omega_0\cos(\tb t+\gamma)\right)\left(x_1(2\pi)-\frac{\tb s_f}{e^{-2\pi\alpha}}\right)\\& + \left(y_{21}(2\pi)\sin(\tb t)+y_{11}(2\pi)\omega_0\cos(\tb t+\gamma)\right)\left(x_2(2\pi)-\frac{\tb v_f}{e^{-2\pi\alpha}}\right) \Bigg),
\end{align*}
where $\alpha=\omega_0\zeta, \ \tb=\omega_0\sqrt{1-\zeta^2}$ and $y_{ij}$ are the components of $y(2\pi)$ given in \eqref{eqn:y_PDHO3}.
\end{corollary}
\begin{proof}
See \hyperlink{cor:PDHO3_proof}{proof} in \ref{sec:proofs}.
\end{proof}
\begin{remark}
Note that Corollary~\ref{cor:projA_PHO} cannot be recovered from Corollary~\ref{cor:projA_PDHO3} by simply taking $\zeta\to0$. Similarly, Corollary~\ref{cor:projA_PDHO1} cannot be recovered from Corollary~\ref{cor:projA_PDHO2} or Corollary~\ref{cor:projA_PDHO3} by taking $\zeta\to1$.
\end{remark}

\subsection{Machine tool manipulator}
A machine tool manipulator is an automatic machine that simulates human hand operations. The dynamics of this machine can be formulated as a LQ control problem as in \cite{MTM}---also see \cite{BurKayMaj2014}. For this problem the system and control matrices in \eqref{A_gen} become
\begin{align*}
A(t) &= A = \begin{bmatrix}
0 & 0 & 0 & 1 & 0 & 0 & 0 \\
0 & 0 & 0 & 0 & 1 & 0 & 0 \\
0 & 0 & 0 & 0 & 0 & 1 & 0 \\
-4.441\times10^7/450 & 0 & 0 & -8500/450 & 0 & 0 & -1/450 \\
0 & 0 & 0 & 0 & 0 & 0 & 1/750 \\
0 & 0 & -8.2\times10^6/40 & 0 & 0 & -1800/40 & 0.25/40 \\
0 & 0 & 0 & 0 & 0 & 0 & -1/0.0025
\end{bmatrix}, \\ B(t) &= b = \begin{bmatrix}
0 & 0 & 0 & 0 & 0 & 0 & 1/0.0025
\end{bmatrix}^T.
\end{align*}
Unlike the special cases in the previous subsection we will not provide analytical projectors for this problem. Because this problem has 7 state variables computing $\Phi_A$ and $J_\phi(0)$ analytically is not a simple task. Instead we will introduce and implement the numerical procedure in Section \ref{sec:num}.

\section{Douglas--Rachford Algorithm}
\label{sec:projAlg}
The Douglas--Rachford algorithm, in our context, is a projection algorithm, which we recall here by  closely following the framework in \cite{BausBuraKaya2019}. We consider a real Hilbert space denoted by $X$, with inner product $\langle\cdot,\cdot\rangle$ and induced norm $\|\cdot\|$.  We will consider the sets $\cal{A}$ and $\cal{B}$ to align with the previous results but note that the only assumptions required are that $\cal{A}$ is a closed affine subspace of $X$ and $\cal{B}$ is a nonempty closed convex subset of $X$.

In our setting, we assume that we are able to compute the projector operators $P_{\cal{A}}$ and $P_{\cal{B}}$.  These operators project a given point onto each of the constraint sets $\cal{A}$ and $\cal{B}$, respectively. Recall that the {\em proximal mapping} of a functional $h$ is defined by \cite[Definition~12.23]{BauCombettes}:
	\begin{equation*}  \label{def:prox}
		{\rm Prox}_h(u) := \argmin_{y\in L^2([t_0,t_f];\dR^m)}
		\left(h(y) + \frac{1}{2}\|y - u\|_{L^2}^2  \right),
	\end{equation*}
	for any $u\in L^2([t_0,t_f];\dR^m)$. We also recall that the {\em indicator function} $\iota_C$ of $C$ is given by
  \[
\iota_{\cal C}(x) := \left\{\begin{array}{ll}
    0\,, & \mbox{ if\ \ } x\in{\cal C}\,, \\
    \infty\,, & \mbox{ otherwise}.
    \end{array}\right.
 \]
Note that ${\rm Prox}_{\iota_{\cal C}}=P_{\cal C}$.
Given $\beta>0$, we specialize the DR algorithm (see
\cite{DougRach}, \cite{LM} and \cite{EckBer}) 
to the case of minimizing the sum of the two functions $f(x):=\iota_{\mathcal{B}}(x) +
\tfrac{\beta}{2}\|x-z\|^2$ and $g :=
\iota_{\cal A}$. For this case, the DR operator becomes
\[ 
T :=  \Id -  {\rm Prox}_f +  {\rm Prox}_g(2 {\rm Prox}_f-\Id). 
\]
Given $f,g$ we know that the respective proximal mappings are 
$ {\rm Prox}_f(x) = P_{\cal{B}}\big(\tfrac{1}{1+\beta}x+\tfrac{\beta}{1+\beta}z\big)$ and
$ {\rm Prox}_g =P_{\cal{A}}$ (see \cite[Proposition~24.8(i)]{BauCombettes}).
Set $\lambda := \tfrac{1}{1+\beta}\in\left]0,1\right[$. 
It follows that the DR operator becomes
\begin{align}
Tx &= x-P_{\cal{B}}\big(\lambda x+(1-\lambda)z\big)+P_{\cal{A}}\Big(2P_{\cal{B}}\big(\lambda
x+(1-\lambda)z\big)-x\Big).
\end{align}
Now fix $x_0\in X$ and let $z:=0$. Given $x_n\in X$, $n\geq 0$, update
\begin{equation}
\label{e:180304a}
b_n:= P_{\cal{B}}\big(\lambda x_n\big),\;\;
x_{n+1} := Tx_n 
= x_n-b_n
+P_{\cal{A}}\big(2b_n-x_n\big).
\end{equation}

Using \cite[Corollary 28.3(v)(a)]{BauCombettes} we have that $(b_n)_{n\in\mathbb{N}}$ converges strongly to the unique solution of Problem (P). Observe that strong convergence is due to the strong convexity of the function $f$, and is for the sequence $(b_n)_{n\in\mathbb{N}}$ and not necessarily $(x_n)_{n\in\mathbb{N}}$. By definition of Problem (P), the unique solution of (P) is the element of minimum norm in $\cal{A}\cap \cal{B}$.  Namely, we have that the limit of the sequence $(b_n)_{n\in\mathbb{N}}$ is $x=P_{\cal{A}\cap \cal{B}}(0)$.  We point out that, in general, only weak convergence is guaranteed for this method (see \cite[Theorem 1]{Svaiter} or \cite[Theorem 4.4]{BauMoursi}).

Note that \eqref{e:180304a} simplifies to 
\begin{equation}\label{eq:DR}
x_{n+1} := x_n - P_{\cal{B}}(\lambda x_n)+P_{\cal{A}}\big(2P_{\cal{B}}(\lambda x_n)-x_n\big)
\quad\text{provided that $z=0$.}
\end{equation}
See Algorithm \ref{alg:DR} below for a step-by-step description of the numerical implementation. \\

\noindent
\begin{algorithm}{({\bf DR})}\label{alg:DR}
\begin{description}
\item[Step 1] ({\em Initialization}) Choose a parameter $\lambda\in\left]0,1\right[$ and the initial iterate $u^0$ arbitrarily. 
Choose a small parameter $\varepsilon>0$, and set $k=0$. 
\item[Step 2] ({\em Projection onto ${\cal B}$})  Set $u^- = \lambda u^{k}$. 
Compute $\widetilde{u} = P_{{\cal B}}(u^-)$ by using \eqref{eqn:projB}. 
\item[Step 3] ({\em Projection onto ${\cal A}$}) Set $u^- := 2\widetilde{u}-u^k$. 
Compute $\widehat{u} = P_{{\cal A}}(u^-)$ by using \eqref{eqn:projA} or Algorithm~\ref{alg:projA}.
\item[Step 4] ({\em Update}) Set $u^{k+1} := u^k + \widehat{u} - \widetilde{u}$.
\item[Step 5] ({\em Stopping criterion}) If $\|u^{k+1} - u^k\|_{L^\infty} \le \varepsilon$, then return $\widetilde{u}$ and stop.  
Otherwise, set $k := k+1$ and go to Step 2.
\end{description}
\end{algorithm}


\begin{remark} \rm
Robustness of the DR algorithm is supported by the fact that many inexact versions of it are shown to converge as well, see \cite{Svaiter2019,AlvEckGerMel2020}. In \cite{AlvGer2019} we see a study of the complexity of an inexact version of the algorithm. This justifies the use of discrete approximations of the function iterates in our implementation.
\proofbox
\end{remark}

\begin{remark} \rm
The convergence properties of the DR algorithm for the case when ${\cal A} \cap {\cal B} = \emptyset$, i.e., when Problem~(P) is infeasible, have been studied recently by Bauschke and Moursi~\cite{BauMou2021}.  A study of this interesting case for optimal control might be an promising direction to pursue in the future.
\proofbox
\end{remark}

\section{Numerical Approach}
\label{sec:num}

In Subsection \ref{subsec:proj_2var} we have a selection of problems where we have derived analytical expressions for the projection onto $\cal{A}$. In practice however the state transition matrix may be too difficult (if not impossible) to find analytically, in which case one needs to employ a numerical technique, as will be outlined further below. Following the presentation of this algorithm we give numerical experiments used to choose the optimal values of the parameter $\lambda$ for the DR algorithm and compare the performance of DR with the AMPL-Ipopt suite.

\subsection{Background and algorithm for projector onto ${\bf\cal A}$}
From Equation \eqref{eqn:gen_u} we can see that in order to define the projection we must find $\lambda$. In Theorem \ref{thm:projA_gen} we assumed that $\Phi_{A}(t_0,t)$ is available. We can see from \eqref{eqn:lambda_gen} that the knowledge of $\Phi_{A}(t_0,t)$ is necessary to find $\lambda$. In the case where we cannot find the state transition matrix directly to substitute into \eqref{eqn:lambda_gen}, we must solve 
\begin{align}\label{eqn:lin_sys}
\begin{bmatrix}\dot{x}(t) \\ \dot{\lambda}(t)\end{bmatrix} = \begin{bmatrix}A(t) & -B(t)B^T(t) \\ 0_{n\times n} & -A^T(t)\end{bmatrix}\begin{bmatrix}x(t) \\ \lambda(t)\end{bmatrix} + \begin{bmatrix}B(t) \\ 0_{n\times m}\end{bmatrix}u^-(t)\,,
\end{align}
for all $t\in[t_0,t_f]$, with the initial conditions (ICs) $x(t_0)=x_0$ and $x(t_f)=x_f$, to find $\lambda$. 
Throughout the steps of Algorithm~\ref{alg:projA}, we will solve the linear system \eqref{eqn:lin_sys} with different ICs.
The ICs that we will consider are
\begin{align}\label{eqn:IC}
\mbox{(i)} \begin{bmatrix}x(t_0) \\ \lambda(t_0)\end{bmatrix} = \begin{bmatrix}x_0 
\\ 0 \end{bmatrix}, \ \
\mbox{(ii)} \begin{bmatrix}x(t_0) \\ \lambda(t_0)\end{bmatrix} = \begin{bmatrix}x_0 \\ e_i \end{bmatrix}, \ \
\mbox{(iii)}\begin{bmatrix}x(t_0) \\ \lambda(t_0)\end{bmatrix} = \begin{bmatrix}x_0 \\ \lambda_0 \end{bmatrix}.
\end{align}

As in the proof of Theorem \ref{thm:projA_gen} we define $z(t,\lambda_0):=x(t)$. Recall in this case that\linebreak $\dot{x}(t) = dx(t)/dt$ can be written as $\partial z(t,\lambda_0)/\partial t$. 
We also recall that the near-miss function $\varphi:\mathbb{R}^n\rightarrow\mathbb{R}^n$ as defined in \eqref{eqn:nearmiss} is affine by Lemma \ref{lem:shooting}(i) and Definition \ref{def:NMF}.  Then the Taylor series expansion of $\varphi$ about zero is simply
\begin{align*}
\varphi(\lambda_0) = \varphi(0)+J_{\varphi}(0)\lambda_0\,.
\end{align*}
Substituting \eqref{eqn:nearmiss}, one gets
\[
z(t_f,\lambda_0) = z(t_f,0)+J_{\varphi}(0)\lambda_0\,,
\]
and, rearranging,
\[
J_{\varphi}(0)\lambda_0 = z(t_f,\lambda_0)-z(t_f,0)\,.
\]
Suppose $\lambda_0=e_i$. Then
\begin{align}\label{eqn:J_comp}
J_{\varphi}(0)e_i= z(t_f,e_i)-z(t_f,0)\,,
\end{align}
which is the $i$th column of $J_{\varphi}(0)$.  Therefore,
by finding $z(t_f,0)$ and $z(t_f,e_i)$ for every $i = 1,\ldots,n$ we can build the Jacobian $J_\varphi(0)$.  Consequently, a procedure for constructing $J_\varphi(0)$ can be prescribed as follows.
\begin{enumerate}
    \item Solve \eqref{eqn:lin_sys} with ICs (ii) in~\eqref{eqn:IC} to get $z(t_f,e_i)$.
    \item Solve \eqref{eqn:lin_sys} with ICs (i) in~\eqref{eqn:IC} to get $z(t_f,0)$.
    \item Compute the $i$th column of $J_\varphi(0)$ using \eqref{eqn:J_comp}, for $i = 1,\ldots,n$, and obtain $J_\varphi(0)$.
\end{enumerate}

As in the proof of Lemma \ref{lem:shooting}, we can now solve the linear system
\begin{equation}  \label{eqn:lineq_lambda0}
 J_{\varphi}(0)\, {\lambda_0} = -\varphi(0) = -(z(t_f,0)-x_f)\,,
\end{equation}
for $\lambda_0$, since in  the procedure for finding $J_\varphi(0)$ we have computed all the other components of this equation. Then once we have $\lambda_0$ we can solve \eqref{eqn:lin_sys} with ICs (iii) in~\eqref{eqn:IC} to obtain $\lambda$.

The algorithm below describes the steps for computing the projection of a current iterate $u^-$ onto the constraint set $\cal{A}$.  In solving \eqref{eqn:lin_sys} with each of the ICs in \eqref{eqn:IC} we implement {\sc Matlab}'s numerical ODE solver {\sf ode45} or a direct implementation of some Runge--Kutta method such as the Euler method.

\begin{algorithm}{({\bf Numerical Computation of the Projector onto ${\cal A}$})} \label{alg:projA}\
\begin{description}
\item[Step 0] ({\em Initialization}) The following are given: Current iterate $u^-$, the system and control matrices $A(t)$ and $B(t)$, the numbers of state and control variables $n$ and $m$, and the initial and terminal states $x_0$ and $x_f$, respectively.
\item[Step 1] ({\em Near-miss function}) Solve \eqref{eqn:lin_sys} with ICs in \eqref{eqn:IC}(i) to find $z(t_f,0) := x(t_f)$.  \\ Set $\varphi(0) :=  z(t_f,0)-x_f$.
\item[Step 2] ({\em Jacobian}) For $i = 1,\ldots,n$, solve \eqref{eqn:lin_sys} with ICs in \eqref{eqn:IC}(ii), to get $z(t_f,e_i)$. \\ 
Set $\beta_i(t) := z(t_f,e_i) - z(t_f,0)$ and $J_\varphi(0) := \left[\beta_1(t)\ |\ \dots\ |\ \beta_n(t) \right]$.
\item[Step 3] ({\em Missing IC}) Solve $J_{\varphi}(0)\,\lambda_0 := -\varphi(0)$ for $\lambda_0$. 
\item[Step 4] ({\em Projector onto ${\cal A}$}) Solve \eqref{eqn:lin_sys} with ICs in \eqref{eqn:IC}(iii) to find $\lambda(t)$.  \\ 
Set  $P_{\cal{A}}(u^-)(t) := u^-(t)-B^T(t)\lambda(t)$.
\end{description}
\end{algorithm}

\subsection{Experiments}
For computations in this section we use {\sc Matlab} release R2021b for implementing the DR algorithm and error analysis. We also use AMPL--Ipopt computational suite \cite{AMPL,WacBie2006} (with Ipopt version 3.12.13) for comparison with the DR algorithm since the suit is commonly used for solving similar optimal control problems. 

In Figure \ref{fig:plots} we have the pure and under-damped oscillator solution plots for the constrained control where $\omega_0=5$.  The boundary conditions are $x_2(0) = 1$ and $x_1(0) = x_1(2\pi) = x_2(2\pi) = 0$.  The bound on $u$ for the under-damped case is much smaller than the value used in the pure case to ensure that the control constraint is active.  In Figure~\ref{fig:MTM}, we display the control variable solution plot for the machine tool manipulator with $|u(t)|\leq 2000$.

\begin{figure}[t!]
\begin{subfigure}{0.5\textwidth}
    \centering
    \includegraphics[width=7cm]{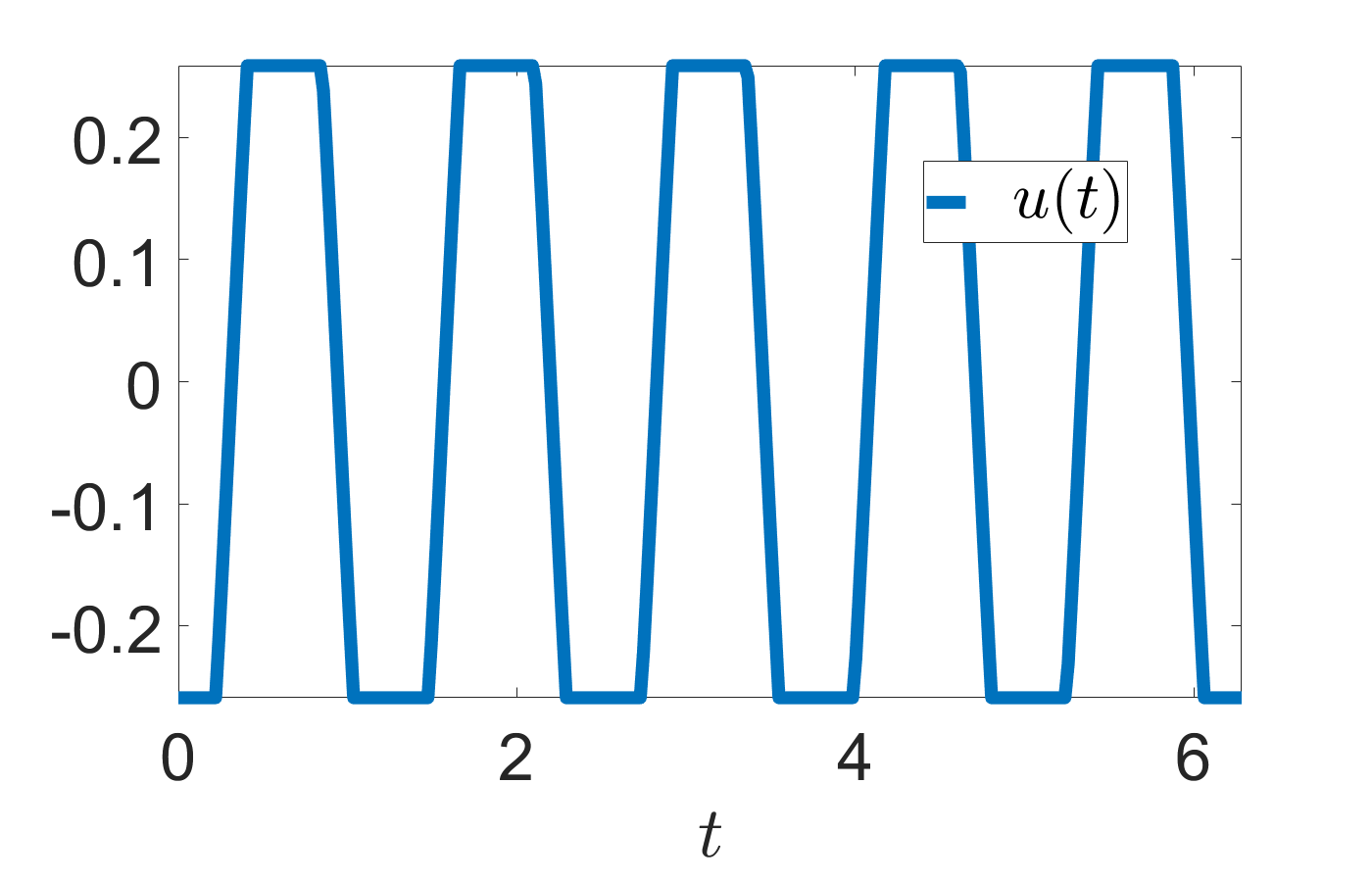}
    \includegraphics[width=7cm]{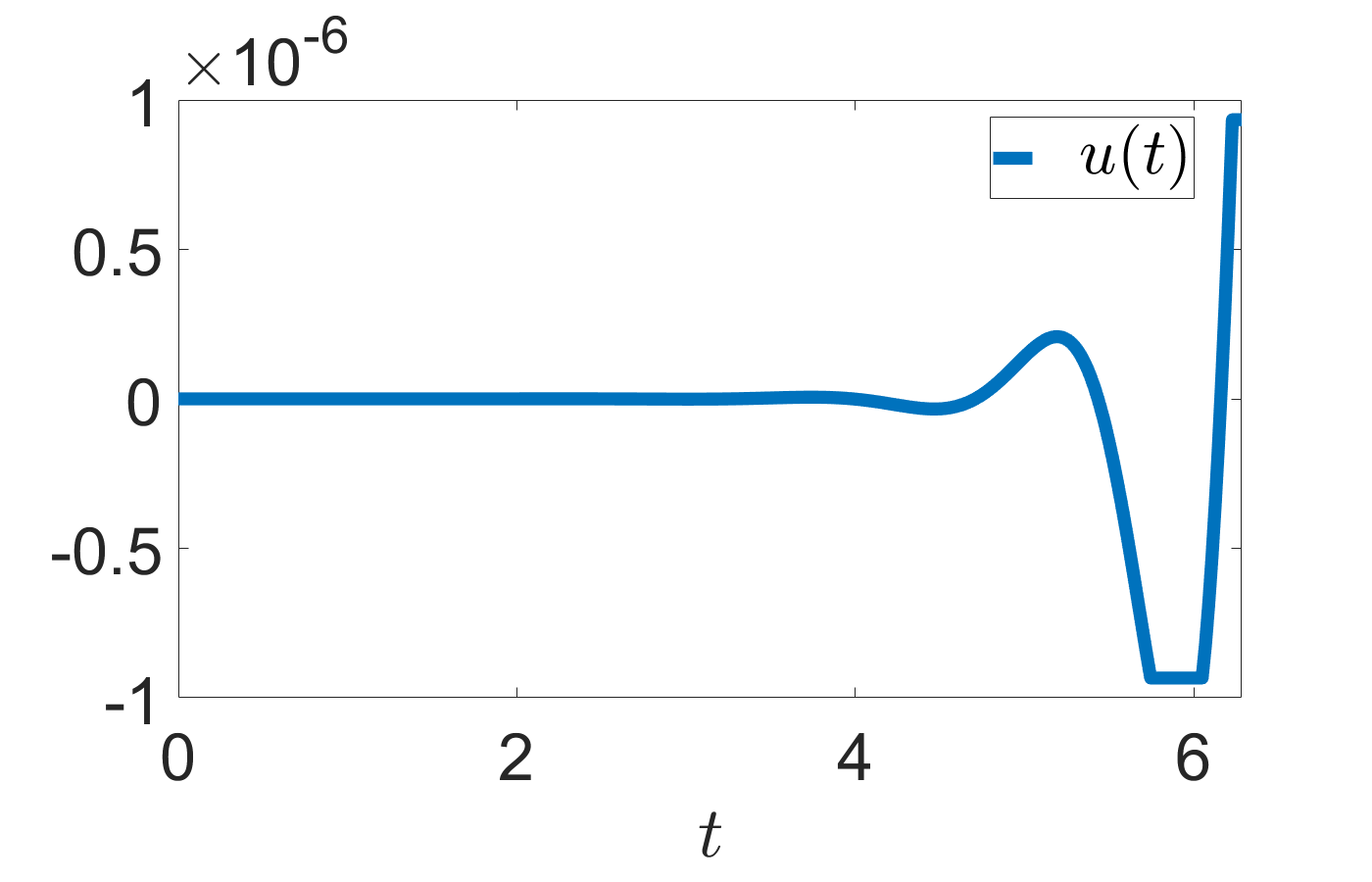}
    \caption{\footnotesize Control $u(t)$.}
\end{subfigure}
\begin{subfigure}{0.5\textwidth}
    \centering
    \includegraphics[width=7cm]{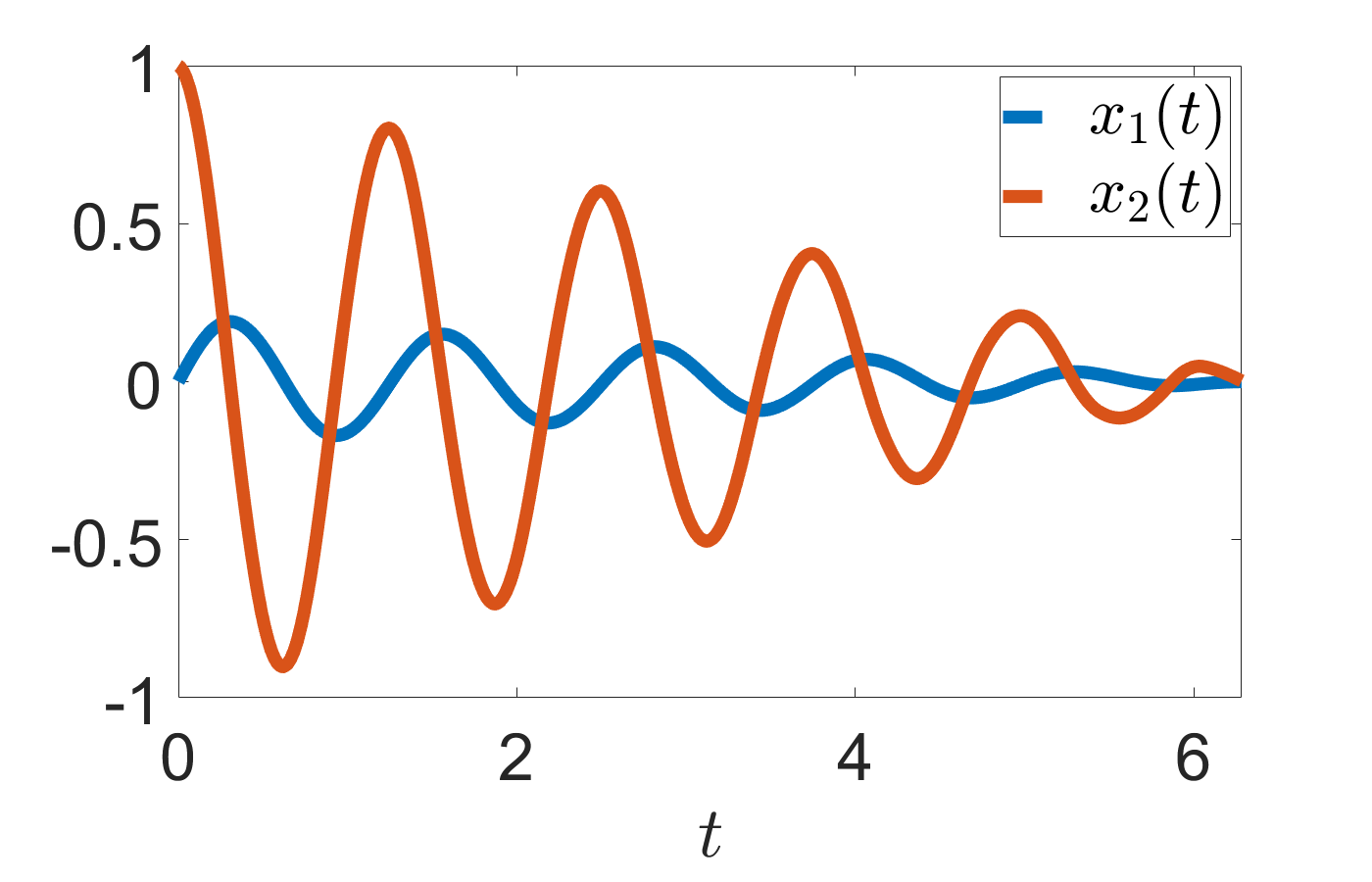}
    \includegraphics[width=7cm]{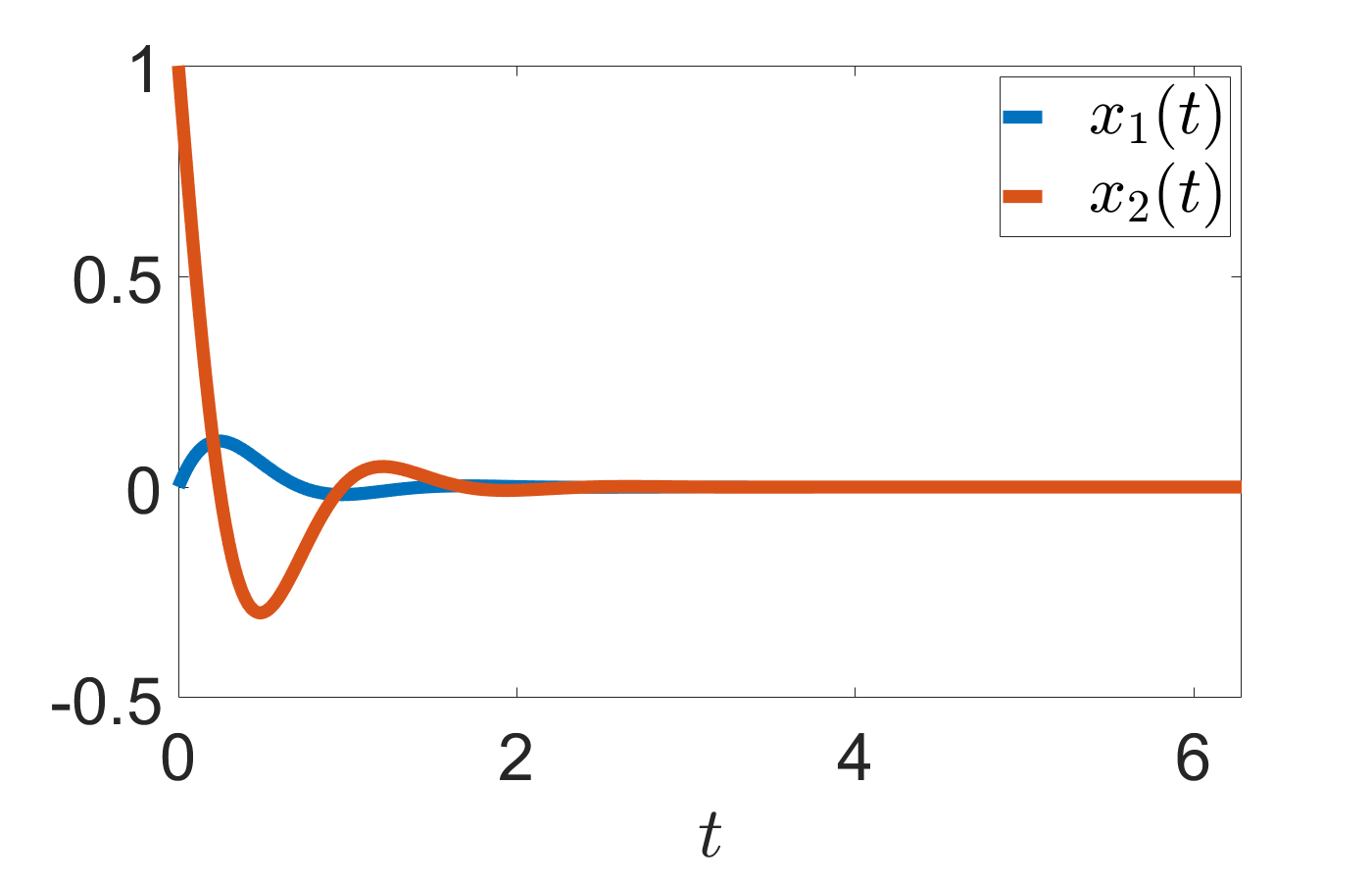}
    \caption{\footnotesize State $x_1(t),x_2(t)$.}
\end{subfigure}
\caption{\sf\small Top plots for $\omega_0=5,\,\zeta=0$ where $|u(t)|\leq 0.259$. Bottom plots for $\omega_0=5,\,\zeta=0.5$ where $|u(t)|\leq 9.34\times10^{-7}$.}
\label{fig:plots}
\end{figure}

\begin{figure}[t!]
    \centering
    \includegraphics[width=10cm]{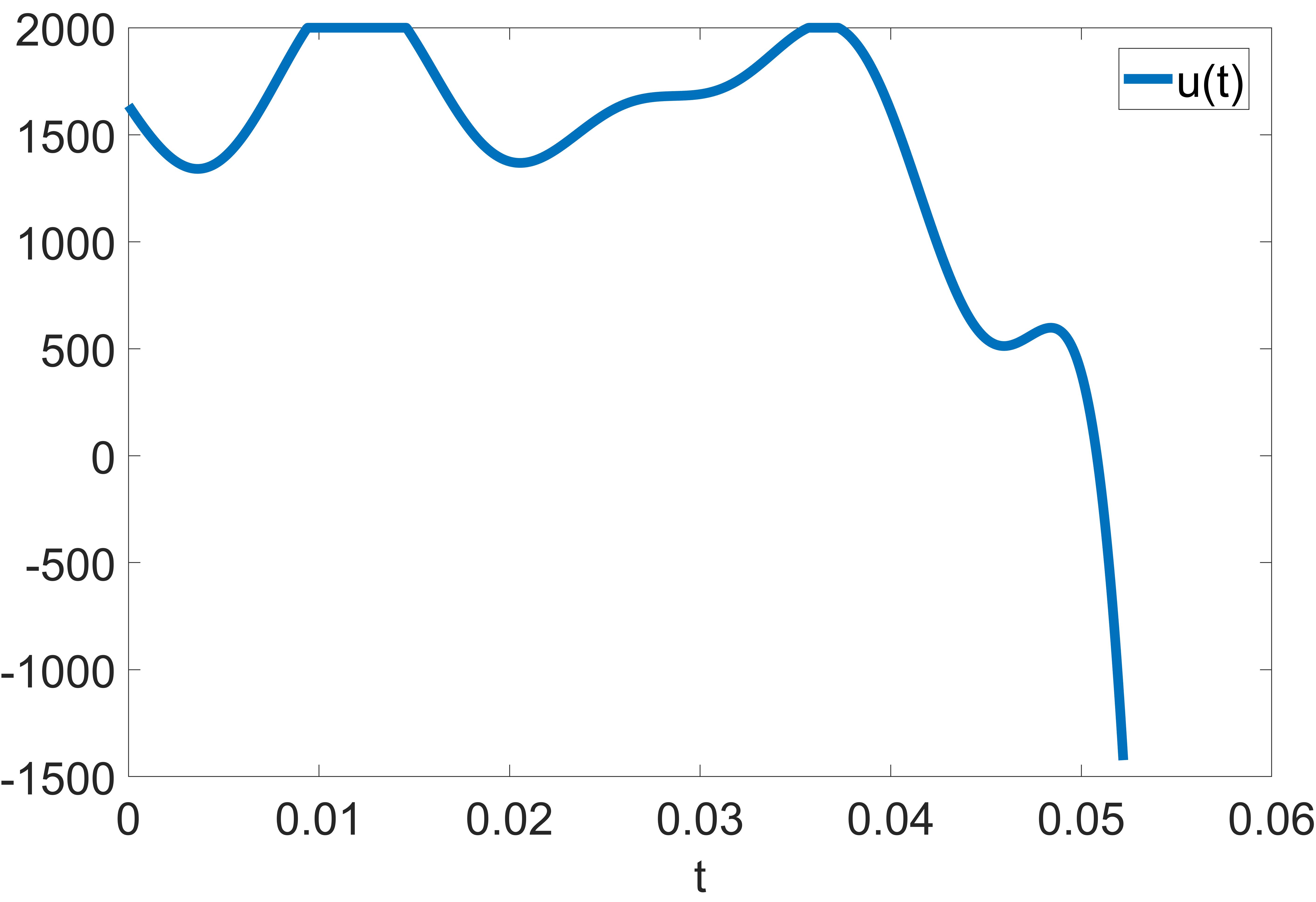}
\caption{\sf\small Control solution plot for the machine tool manipulator where $|u(t)|\leq 2000$.}
\label{fig:MTM}
\end{figure}

\subsubsection{Parameter plots}

In Section \ref{sec:projAlg} the DR algorithm requires a choice of $\lambda\in]0,1[$.   In Figure \ref{fig:PHO} we experiment with different parameter choices for the pure harmonic oscillator, under-damped harmonic oscillator and machine tool manipulator.

In Figure~\ref{fig:PHO} we see, for different bounds on $u$, plots with the number of iterations taken by DR for different values of the parameter $\lambda$. In each of these plots five values for the bound on $u$ were taken, the smallest bound being close to the value that will lead to a problem with no solution and the largest resulting in an unconstrained $u$. These plots give information on the ``best" value of $\lambda$ to choose that produce the smallest number of iterations. This is advantageous because a reduction in the number of iterations will result in a reduction in run time.

\begin{figure}[t!]
\begin{subfigure}{0.5\textwidth}
    \centering
    \includegraphics[width=7cm]{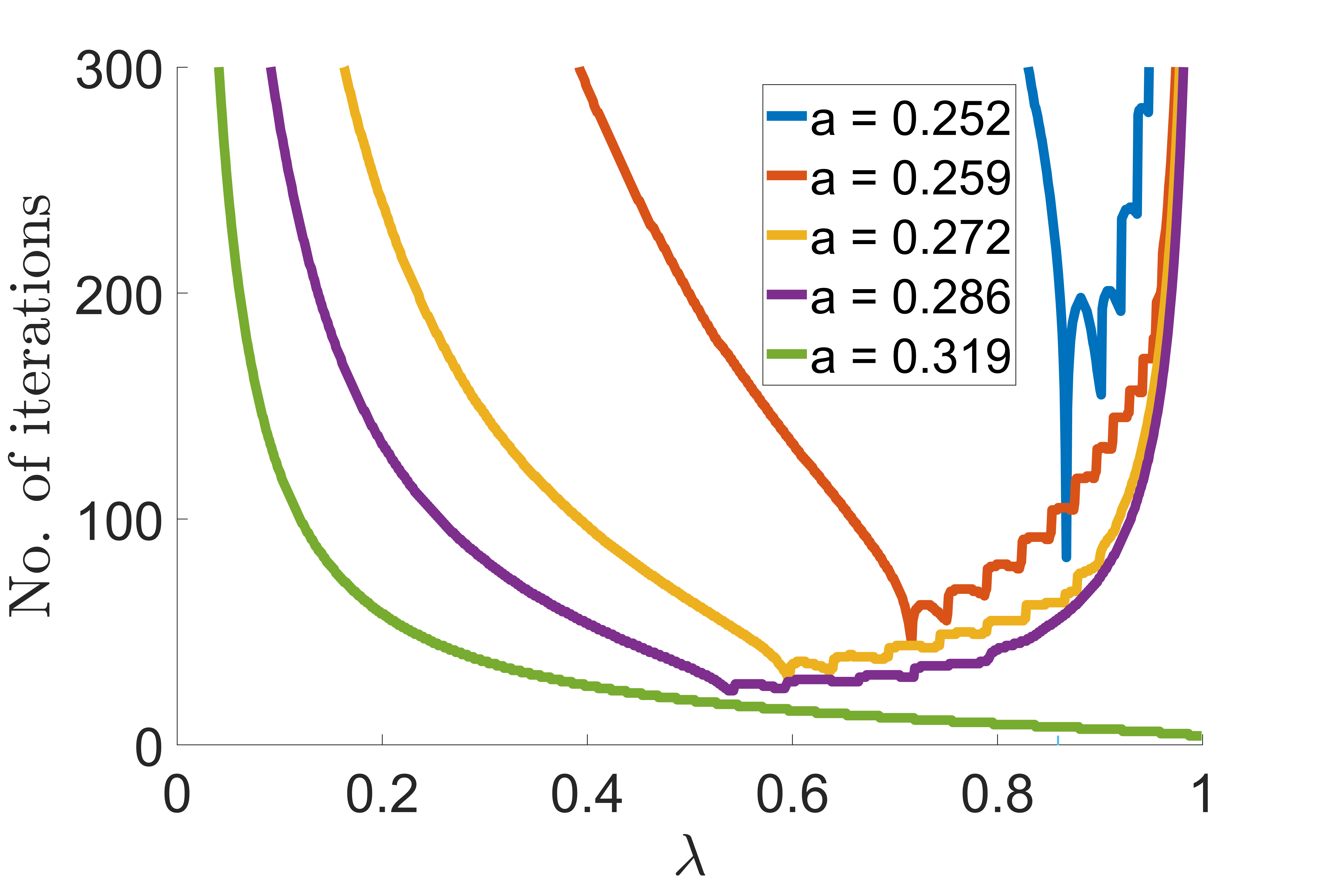} \\
    \caption{\footnotesize $1\le\omega_0\le100,\,\zeta=0,\,\varepsilon=10^{-6}$}
    \label{fig:zeta_zero}
\end{subfigure}
\begin{subfigure}{0.5\textwidth}
    \centering
    \includegraphics[width=7cm]{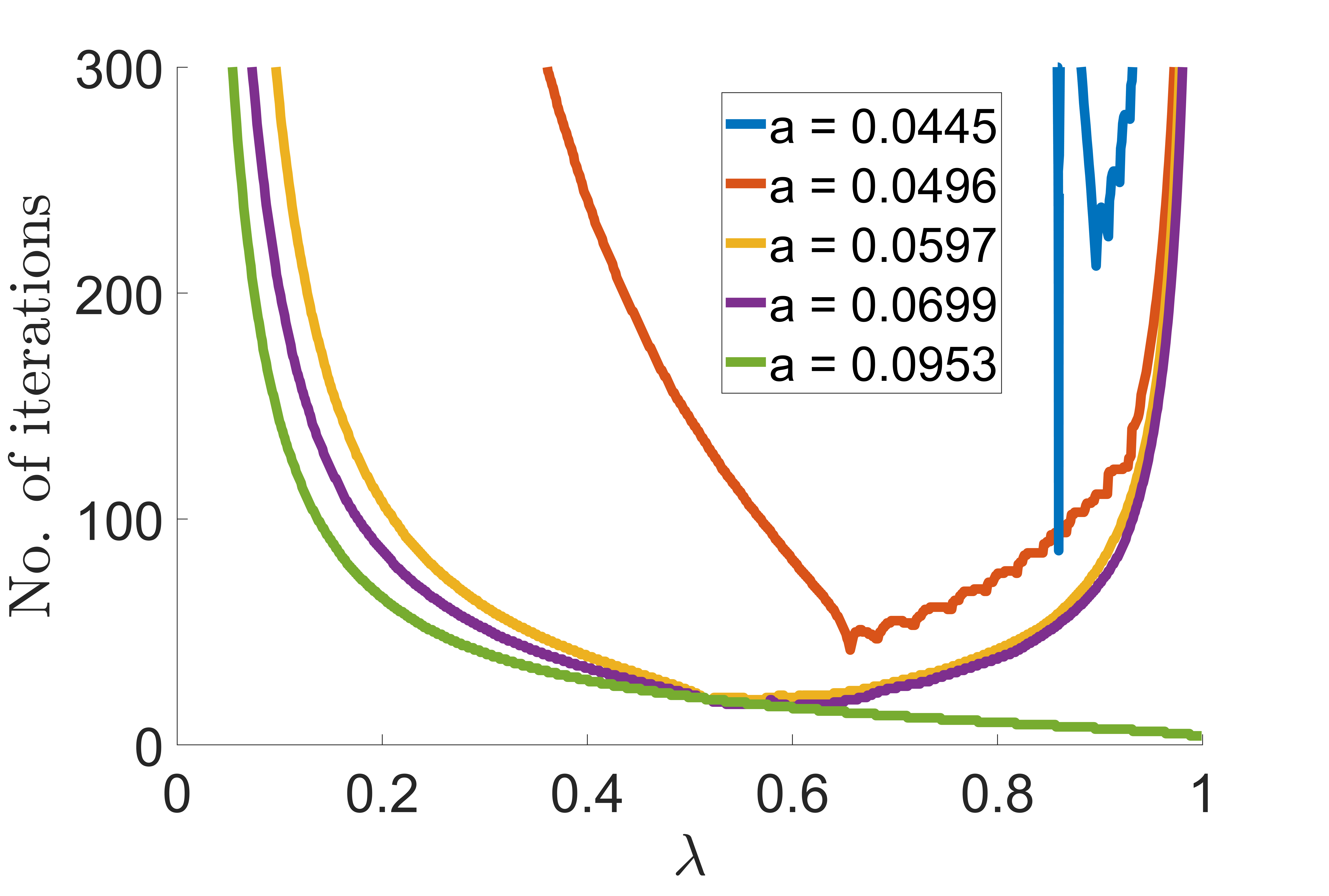} \\
    \caption{\footnotesize $\omega_0=1,\,\zeta=0.5,\,\varepsilon=10^{-7}$}
    \label{fig:w0_1}
\end{subfigure}
\begin{subfigure}{0.5\textwidth}
    \centering
    \includegraphics[width=7cm]{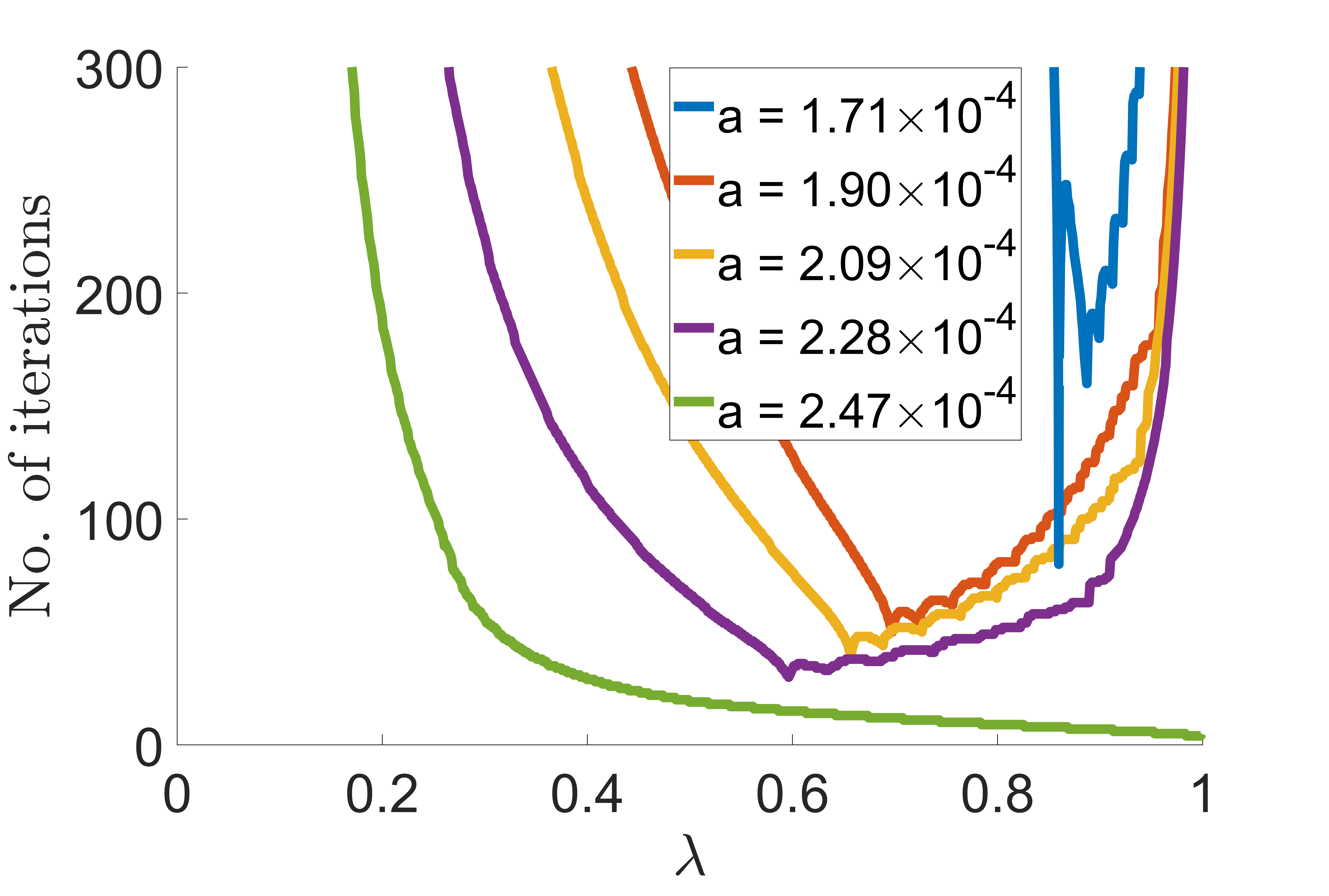} \\
    \caption{\footnotesize $\omega_0=3,\,\zeta=0.5,\,\varepsilon=10^{-9}$}
    \label{fig:w0_3}
\end{subfigure}
\begin{subfigure}{0.5\textwidth}
    \centering
    \includegraphics[width=7cm]{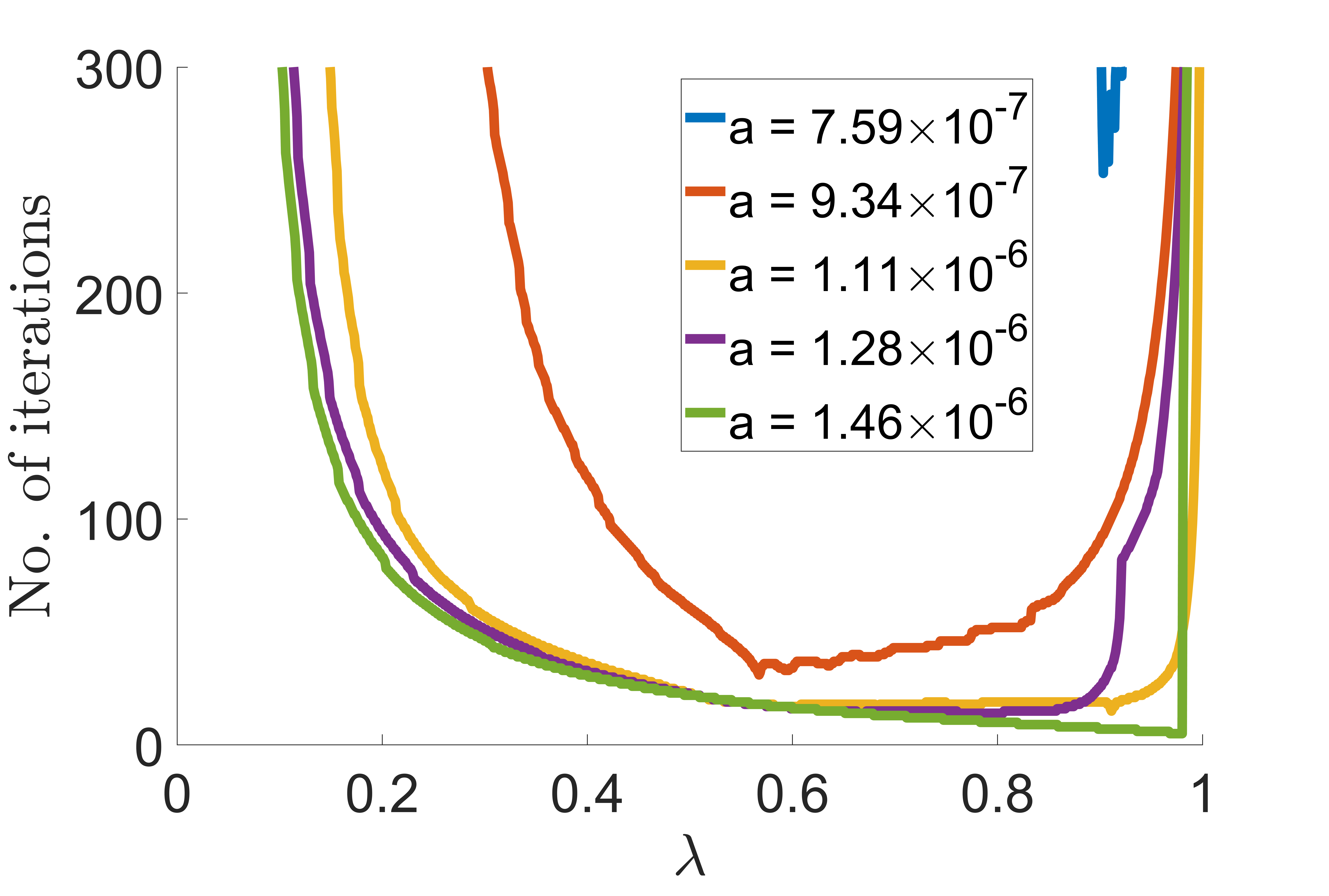} \\
    \caption{\footnotesize $\omega_0=5,\,\zeta=0.5,\,\varepsilon=10^{-12}$}
    \label{fig:w0_5}
\end{subfigure}
\begin{subfigure}{0.5\textwidth}
    \centering
    \includegraphics[width=7cm]{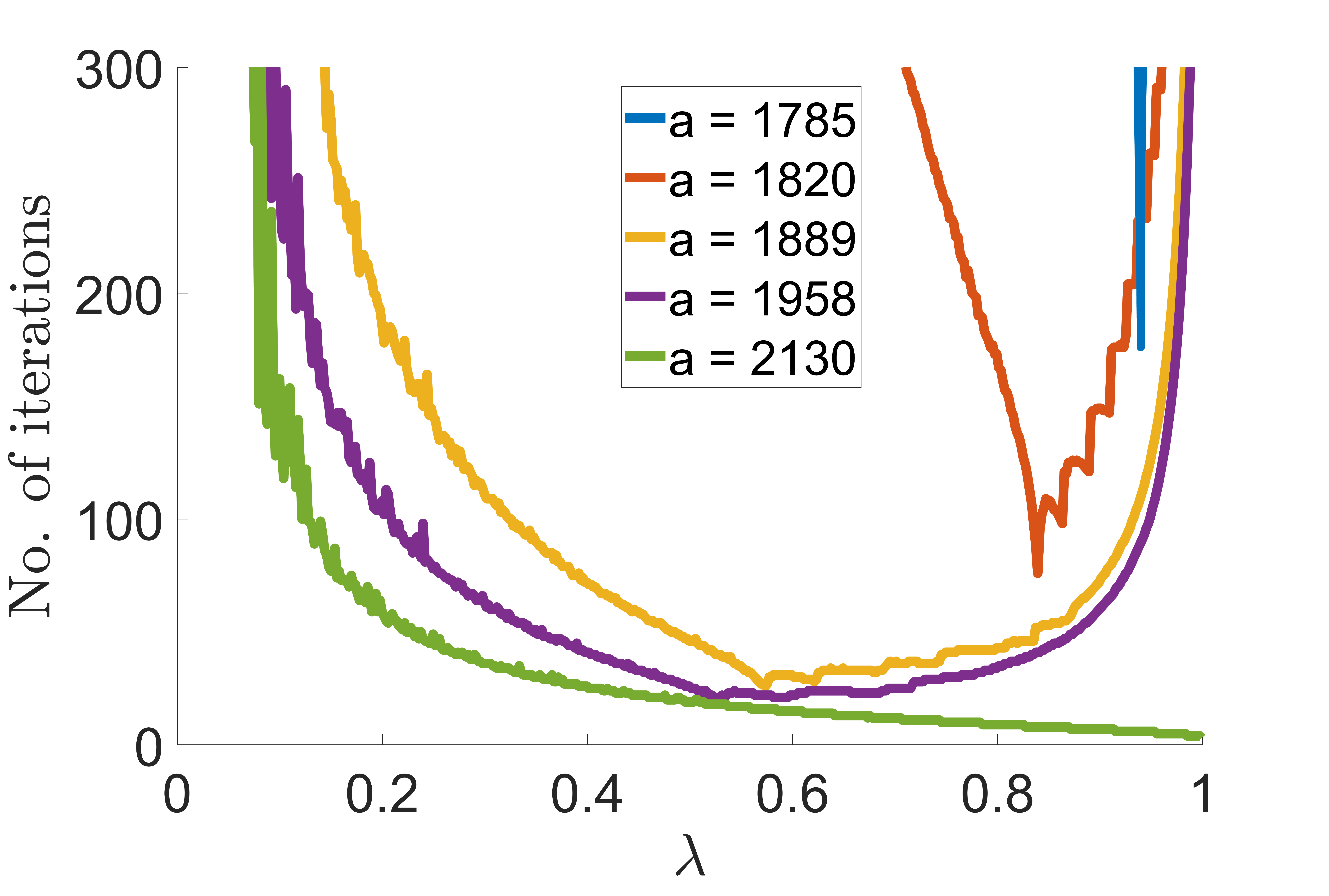} \\
    \caption{\sf\footnotesize Machine tool manipulator with $\varepsilon=10^{-2}$} \label{curve:manip}
\end{subfigure}
\begin{subfigure}{0.5\textwidth}
    \centering
\end{subfigure}
\caption{\sf\small Parameter curves for the harmonic oscillator with various values of $(\omega_0,\zeta)$ and for the machine tool manipulator with tolerance values.}
\label{fig:PHO}
\end{figure}

Figure \ref{fig:PHO} contains the experiments for the harmonic oscillator with various values of the pair $(\omega_0,\zeta)$, as well as the machine tool manipulator. We see from Figure \ref{fig:zeta_zero} that when $\zeta = 0$ varying $\omega_0$ does not seem to have an impact on the best choice for $\lambda$, in fact the curves numerically found over the values $1 \le \omega_0 \le 100$ seem to be identical. For example when $|u(t)|\leq0.259$ the ``best" value for $\lambda$ looks to be approximately $0.7$ for any $1 \le \omega_0 \le 100$.  We observe downward spikes at some parameter values that result in a large decrease in iterations. Although these spikes achieve a much smaller number of iterations we would not necessarily select these values in practice because a slight shift from the parameter values results in a large increase in the number of iterations.

In Figures \ref{fig:w0_1}--\ref{fig:w0_5} we have the parameter graphs for the under-damped harmonic oscillator with $\zeta=0.5$ and $\omega_0=1,3$ and $5$.  This time there are differences in the curves for different values of $\omega_0$. For this damped problem the larger the value of $\omega_0$, the closer the optimal $u$ gets to zero; see Figure~\ref{fig:plots}. Hence the values of $a$ are chosen to be much smaller for $\omega_0=3,5$ so the control constraint remains active though again the largest value of $a$, given by the enveloping lowermost (green) curve, is the case of unconstrained $u$. 

The behaviour of the cases with $\omega_0=1$ and $3$ are very similar with both having some spikes present, the optimal $\lambda$ value for the unconstrained $u$ case being almost 1 and bumps present in the curves. Though for $\omega_0=5$ we see some odd behaviour for the three largest values of~$a$. When using the lowermost (green) curve, i.e., when $u$ in unconstrained, the number of iterations greatly increases at $\lambda\approx0.98$. We also observe for the yellow and purple curves that the number of iterations seems to 
level off before rapidly increasing which is something we don't observe in the other cases. These anomalies could potentially be numerical artifacts but further investigation is needed to draw a conclusion.

For the machine tool manipulator in Figure~\ref{curve:manip} we again see a lot of similarities to the figures from the other problems. Although for the machine tool manipulator we see many small ripples in the curves and for the blue curve there are almost no parameter choices where the methods converge in less than 300 iterations. The blue curves represent the problem where $a$ is so small that there are almost no solutions to the problem.  So it is fair to say that when the problem is almost infeasible it is impossible to get a solution in reasonable time.

In general we observe some similarities across all the problems in Figure \ref{fig:PHO}. For all the problems the ``best" parameter choices are $\lambda\geq0.5$. As we approach the critical value of $a$ where the problems have no solution the ``best'' choice of $\lambda$ approaches 1.

\subsubsection{Error and CPU time comparisons}

The iterates of the DR algorithm (Algorithm~\ref{alg:DR}) are functions and in each iteration function addition and scalar multiplication operations need to be performed.  Obviously we can perform these operations numerically only on approximations of functions. For approximations we consider discretization of the function iterates in that the iterates are represented by $N$ discrete values over a regular partition of their domains.

The AMPL--Ipopt suite is, on the other hand, already a numerical scheme for finite-dimensional optimization problems and as such it is applied to the {\em direct} ({\em Euler}) {\em discretization} (see~e.g.~\cite{Hager2000}) of Problem~(P), with the same $N$ so that the discrete solutions obtained by the DR algorithm and the AMPL--Ipopt suite can be compared.

In Step 5 of Algorithm \ref{alg:DR} we use as stopping criterion the difference between two consecutive iterates in function space. In the implementation of Algorithm \ref{alg:DR}, discretized iterates are used so in turn Step 5 uses finite dimensional norm to evaluate this stopping criterion.

We rather compute {\em a posteriori} the absolute true errors in the solution depending on $N$.  Namely if $u_N$ denotes the approximate (discretized) solution of control and $u_N^*$ the discretized exact/true solution, then the error is the maximum of the absolute difference, in other words, $\|u_N-u_N^*\|_{\ell_\infty}$. The fact that the stopping criterion in Step 5 is effective is shown by the fact that the actual absolute error tends to zero as $N$ grows.

In Table \ref{tbl:erroru} we display these errors as well as the CPU times with the number discretization points $N=10^3,\,10^4$ and $10^5$ for all the previously mentioned problems with the specifications given in Table \ref{tbl:probs}. Since we cannot find analytical solutions for these problems the ``true" solution $u^*$ we are comparing to in this error analysis was computed using the DR algorithm with $N=10^7$ and tolerance $10^{-12}$.

\begin{table}[t!]
    \centering
    \begin{tabular}{ScScScSc}
        $(\omega_0,\zeta)$ & $\varepsilon$ & $a$ & $\lambda$  \\ \hline \hline
        $(1,0)$ & $10^{-6}$ & $0.259$ & $0.75$ \\
        $(5,0)$ & $10^{-6}$ & $0.259$ & $0.75$ \\
        $(1,0.5)$ & $10^{-7}$ & $4.96\times10^{-2}$ & $0.65$\\
        $(5,0.5)$ & $10^{-12}$ & $9.34\times10^{-7}$ & $0.6$ \\
        MTM & $10^{-2}$ & $2000$ & $0.55$
    \end{tabular}
    \caption{Tolerance, bounds on control variable and parameter choices for numerical experiments. MTM stands for machine tool manipulator.}
    \label{tbl:probs}
\end{table}

We are mostly interested in the errors in the control variable since that is the variable being optimized. The states are computed as an auxiliary process using the optimal control found and Euler's method. In Table \ref{tbl:erroru} we see that for the DR algorithm, in general, an increase in the discretization points used by some order results in a decrease in the error by the same order. This is useful to know because if a particular error is required the number of discretization points needed to reach that accuracy can easily be determined.

\begin{table}[t!]
\centering
\small
    \begin{tabular}{ScSclllllll}
 & & \multicolumn{2}{|c|}{$L^\infty$ error in control} & \multicolumn{2}{c|}{$L^\infty$ error in states} & \multicolumn{2}{c}{CPU time [sec]} \\ \cline{3-4}  \cline{5-6}  \cline{7-8}
    $N$ & $(\omega_0,\zeta)$ & \multicolumn{1}{|c}{DR} & \multicolumn{1}{c|}{Ipopt} & \multicolumn{1}{c}{DR} & \multicolumn{1}{c|}{Ipopt} & \multicolumn{1}{c}{DR} & \multicolumn{1}{c}{Ipopt} \\ \hline\hline
    & $(1,0)$ & $4.0\times10^{-3}$ & $4.2\times10^{-2}$ & $1.5\times10^{-2}$ & $1.3\times10^{-2}$ & $4.2\times10^{-3}$ & $2.4\times10^{-1}$ \\
    & $(5,0)$ & $1.8\times10^{-2}$ & \multicolumn{1}{c}{$-$} & $3.5\times10^{-1}$ & \multicolumn{1}{c}{$-$} & $4.5\times10^{-3}$ & \multicolumn{1}{c}{$-$} \\
    $10^3$ & $(1,0.5)$ & $2.1\times10^{-3}$ & $6.8\times10^{-3}$ & $4.3\times10^{-3}$ & $5.1\times10^{-3}$ & $5.5\times10^{-3}$ & $2.3\times10^{-1}$ \\
    & $(5,0.5)$ & $1.2\times10^{-7}$ & $1.5\times10^{-6}$ & $1.5\times10^{-2}$ & $1.5\times10^{-2}$ & $4.3\times10^{-3}$ & $1.9\times10^{-1}$ \\
    & MTM & $9.3\times10^{1}$ & $4.6\times10^{1}$ & $2.8\times10^{1}$ & $2.9\times10^{1}$ & $1.0\times10^{-1}$ & $1.2\times10^{0}$ \\
    \hline
    & $(1,0)$ & $4.0\times10^{-4}$ & $1.4\times10^{-2}$ & $1.5\times10^{-3}$ & $3.1\times10^{-3}$ & $4.9\times10^{-2}$ & $2.1\times10^{0}$ \\
    & $(5,0)$ & $1.8\times10^{-3}$ & $1.2\times10^{-1}$ & $2.6\times10^{-2}$ & $7.5\times10^{-3}$ & $5.4\times10^{-2}$ & $1.3\times10^{-1}$ \\
    $10^4$ & $(1,0.5)$ & $2.1\times10^{-4}$ & $1.0\times10^{-2}$ & $4.3\times10^{-4}$ & $9.3\times10^{-3}$ & $4.7\times10^{-2}$ & $1.9\times10^{0}$ \\
    & $(5,0.5)$ & $1.2\times10^{-8}$ & $8.2\times10^{-7}$ & $1.5\times10^{-3}$ & $1.5\times10^{-3}$ & $4.1\times10^{-2}$ & $1.5\times10^{0}$ \\
    & MTM & $9.2\times10^{0}$ & \multicolumn{1}{c}{$-$} & $2.9\times10^{0}$ & \multicolumn{1}{c}{$-$} & $1.1\times10^{0}$ & \multicolumn{1}{c}{$-$} \\
    \hline
    & $(1,0)$ & $4.0\times10^{-5}$ & $2.5\times10^{-1}$ & $1.5\times10^{-4}$ & $5.8\times10^{-2}$ & $4.2\times10^{-1}$ & $8.5\times10^{1}$ \\
    & $(5,0)$ & $1.7\times10^{-4}$ & $7.6\times10^{-2}$ & $2.5\times10^{-3}$ & $3.1\times10^{-3}$ & $4.8\times10^{-1}$ & $1.5\times10^{0}$ \\
    $10^5$ & $(1,0.5)$ & $2.1\times10^{-5}$ & $1.8\times10^{-2}$ & $4.3\times10^{-5}$ & $1.8\times10^{-2}$ & $4.1\times10^{-1}$ & $1.8\times10^{1}$ \\
    & $(5,0.5)$ & $1.2\times10^{-9}$ & $8.4\times10^{-7}$ & $1.5\times10^{-4}$ & $1.5\times10^{-4}$ & $3.7\times10^{-1}$ & $1.5\times10^{1}$ \\
    & MTM & $5.5\times10^{-1}$ & \multicolumn{1}{c}{$-$} & $2.6\times10^{-1}$ & \multicolumn{1}{c}{$-$} & $9.5\times10^{1}$ & \multicolumn{1}{c}{$-$} \\
    \hline
    \end{tabular}
    \caption{\sf\small Errors in control and states and CPU times for the DR algorithm and AMPL--Ipopt, with specifications from Table \ref{tbl:probs}. For Ipopt we set the tolerance {\tt tol} to $10^{-6}$. A dash means a method was unsuccessful in getting a solution. MTM stands for machine tool manipulator.}
    \label{tbl:erroru}
\end{table}

The only case where this seems to differ is in the machine tool manipulator example. For example with $N=10^3$ the error with the DR algorithm is $9.3\times10^{1}$ so following the observed pattern we would expect that when $N=10^5$ the error should be around $9.3\times10^{-1}$ but instead we have $5.5\times10^{-1}$. A possible explanation may be that the expected error is not far enough from the ``true" solution. Or since the machine tool manipulator is the only example to use the numerical implementation rather than analytical expressions the errors are different to what we expected. Whatever the reason we see that the machine tool manipulator is still achieving at least one order of improvement in error for the control variable.


In Table \ref{tbl:erroru} we also observe that for the states we have the same relationship between the orders of the number of discretization points and error for the DR algorithm.
In the states the errors for the DR algorithm are much closer to that of Ipopt, with Ipopt resulting in better error in some of the cases. These differences in performance of the DR algorithm in the state variables could be because of the extra errors introduced from Euler's method in the computation of the states. If we were to use a more accurate method to compute the states it is possible we would have less discrepancies, though by implementing a more complicated method the run time would increase.

In the fifth and sixth columns of Table~\ref{tbl:erroru} we have the CPU times. In these calculations the CPU times recorded are averages from 1,000 runs on a PC with an i5-10500T 2.30GHz processor and 8GB RAM. On average the DR algorithm is more than 10 times faster, with some cases of the DR algorithm being as much as 200 or more times faster than AMPL--Ipopt. In general for the DR algorithm an increase in the number of discretization points results in a proportional increase in run time. We can use this observation to estimate the CPU time for any number of discretization points.

\section{Conclusion and Open Problems}
\label{sec:con}
We have derived general expressions for the projectors respectively onto the affine set and the box of the minimum-energy control problem. We provided closed-form expressions for the pure, critically-, over- and under-damped harmonic oscillators. For problems where we do not have the necessary information to use the general expression for the projector onto the affine set, we proposed a numerical scheme to compute the projection. In our numerical experiments we have applied this numerical scheme to solve a machine tool manipulator problem. We carried out numerical experiments with all the previously mentioned problems, the closed-form examples and the machine tool manipulator, comparing the errors and CPU times. These numerical experiments compared the performance of the DR algorithm with the AMPL--Ipopt suite.

For the DR algorithm we collected some numerical results regarding the use of different values for the parameter $\lambda$ and its effect on the number of iterations required for the method to converge. In this parameter analysis we observed that as the bounds on the control variable are tightened the choice of parameter becomes more difficult. We also noticed that when the problem is almost infeasible, i.e., the bounds on the control variable are so tight that almost no solutions exist, the parameter value approaches 1.

Regarding our other numerical experiments we observe that an increase in the order of discretization points produces a resulting decrease in the order of the errors, both for the control and state variables. We also see that an increase in the order of discretization points results in an increase in the order of the CPU time. These observations are useful to estimate the run time and errors for any number of discretization points and were not seen in the results from Ipopt. In general we see smaller errors and faster CPU times when using the DR algorithm. Overall, using the DR algorithm with the general expressions and numerical approach we proposed is more advantageous than using Ipopt for the class of problems we consider.


In the future it would be useful to extend this research to more general problems. One such extension is the case when the ODE constraints are nonlinear. If the ODE constraints are nonlinear then we have a nonconvex problem. The DR algorithm has already been shown to have success with finite-dimensional nonconvex problems so it would be interesting to apply this method to nonconvex minimum-energy control problems. An extension to LQ problems where the objective function is given by
\begin{equation*}
\frac{1}{2}\int_{t_0}^{t_f} \Big( x(t)^TQ(t)x(t) + u(t)^TR(t)u(t) \Big)\, dt,
\end{equation*}
where $Q,R$ are positive semi-definite and positive definite matrix functions of dimensions $n\times n$ and $m\times m$, respectively, should also be investigated. The fact that the state variables appear in the objective makes this problem particularly interesting and challenging to study. 

Another possibility is to look into the cause of the intriguing numerical results in Figure~\ref{fig:w0_5} where $\omega_0=5$, $\zeta=0.5$, using the DR algorithm. It is currently unknown whether the behaviour in this case is a result of a theoretical fact or it is just
 a numerical artifact. 

In an earlier preprint version of this paper on arXiv~\cite{BurCalKay2022}, various other projection methods are tested as well as the DR algorithm on the same optimal control problems as in this paper.  These additional methods are namely the method of alternating projections (MAP)~\cite{BauBor1996, vonNeumann1949} and the Dykstra~\cite{BoyleDykstra} and Arag\'on Artacho--Campoy (AAC)~\cite{AAC} algorithms.  MAP consists of sequential, or alternating, projections onto each of the sets $\cal A$ and $\cal B$, and Dykstra is some modification of MAP. However, MAP can only find a point in $\cal A\cap \cal B$. Unlike DR or Dykstra, MAP cannot handle an objective function that is not an indicator. 
Performance comparisons with another projection method, namely the AAC algorithm, a special case of which is the DR algorithm, can also be found in~\cite{BurCalKay2022}.  We chose to focus on a single projection algorithm here, namely the DR algorithm, for which convergence theory involving convex optimization problems in Hilbert spaces have been well studied and cited widely in the literature. In the future it would be interesting to implement other projection algorithms such as the Peaceman--Rachford algorithm or the projected gradient method.

\newpage
\setcounter{section}{0}
\appendix
\renewcommand{\thesection}{Appendix \Alph{section}}
\section{}  \label{sec:proofs}
In this appendix we provide the proofs of the lemmas and corollaries from Section \ref{subsec:proj_2var}.\\[2mm]
\noindent
\textbf{\hypertarget{lem:PDE_proof}{Proof of Lemma \ref{lem:PDE}}}\\[2mm]
With the $A$ given for the double integrator, it is straightforward to compute the expressions in~\eqref{PDI_Phi}.
To find $J_{\varphi}(0)$ we need to solve Equation \eqref{eqn:y} with
\begin{equation*}
\widetilde{A} = \left[\begin{array}{cc;{3pt/3pt}cc}
0 & 1 \ \ & 0 & 0 \\
0 & 0 \ \ & 0 & 0 \\[1mm] \hdashline[3pt/3pt] \\[-4mm]
0 & 0 \ \ & 0 & 1 \\
0 & 0 \ \ & 0 & 0
\end{array}\right] \ \
\widetilde{b} = \left[\begin{array}{c}
0 \\
t \\[1mm] \hdashline[3pt/3pt] \\[-4mm]
0 \\
-1
\end{array}\right].
\end{equation*}
Using Equation \eqref{eqn:y}
\begin{equation*}
y(1) = \left[\begin{array}{c}
y_1(1) \\[1mm] \hdashline[3pt/3pt] \\[-4mm]
y_2(1)
\end{array}\right] =  \bigint_0^1
\left[\begin{array}{cc;{3pt/3pt}cc}
1 & 1-\tau \ & \ 0 & 0 \\
0 & 1 \ & \ 0 & 0 \\[1mm] \hdashline[3pt/3pt] \\[-4mm]
0 & 0 \ & \ 1 & 1-\tau \\
0 & 0 \ & \ 0 & 1
\end{array}\right]
\left[\begin{array}{c}
0 \\
\tau \\[1mm] \hdashline[3pt/3pt] \\[-4mm]
0 \\
-1
\end{array}\right]
\,d\tau
= 
\left[\begin{array}{c}
1/6 \\
1/2 \\[1mm] \hdashline[3pt/3pt] \\[-4mm]
-1/2 \\
-1
\end{array}\right].
\end{equation*}
This gives us
\begin{equation*}
J_{\varphi}(0) = \left[\begin{array}{c;{3pt/3pt}c}
y_1(1) \ & y_2(1)
\end{array}\right] =
\left[\begin{array}{c;{3pt/3pt}c}
1/6 \ \ & -1/2 \\
1/2 \ \ & -1
\end{array}\right]
,
\end{equation*}
when inverted we arrive at the expression in \eqref{PDI_Phi}.
\proofbox

\noindent
\textbf{\hypertarget{cor:PDE_proof}{Proof of Corollary \ref{cor:projA_PDI}}}\\[2mm]
Recall the results for the state transition matrices and the Jacobian in Lemma~\ref{lem:PDE}. By direct substitution of these quantities into~\eqref{eqn:projA_spec}, we get
\begin{align*}
P_{{\cal A}_{0,0}}(u^-)(t) &= u^-(t)+\!\begin{bmatrix}
0 \\ 1
\end{bmatrix}^T\!
\begin{bmatrix}
1 & 0 \\
-t & 1
\end{bmatrix}\!
\begin{bmatrix}
-12 & 6 \\
-6 & 2
\end{bmatrix}\!
\left(\begin{bmatrix}
1 & 1 \\
0 & 1
\end{bmatrix}\!
\begin{bmatrix}
s_0 \\
v_0
\end{bmatrix}\right. \\ & +\left.\int_0^1
\begin{bmatrix}
1 & 1-\tau \\
0 & 1
\end{bmatrix}\!
\begin{bmatrix}
0 \\
u^-(\tau)
\end{bmatrix}d\tau-
\begin{bmatrix}
s_f \\
v_f
\end{bmatrix}\right), \\
&= u^-(t)+\begin{bmatrix}
12t-6 & -6t+2
\end{bmatrix}
\left(\begin{bmatrix}
s_0+v_0-s_f \\ v_0-v_f
\end{bmatrix}+\int_0^1\begin{bmatrix}
(1-\tau)u^-(\tau) \\ u^-(\tau)
\end{bmatrix}d\tau\right), \\
&= u^-(t)+\left(12\left(s_0+v_0-s_f+\int_0^1 (1-\tau)u^-(\tau)d\tau\right)\right.\\& -6\left.\left(v_0-v_f+\int_0^1 u^-(\tau)d\tau\right)\right)t\\&-6\left(s_0+v_0-s_f+\int_0^1 (1-\tau)u^-(\tau)d\tau\right)+2\left(v_0-v_f+\int_0^1 u^-(\tau)d\tau\right),
\end{align*}
as required.
\proofbox

\noindent
\textbf{\hypertarget{lem:PHO1_proof}{Proof of Lemma \ref{lem:PHO}}}\\[2mm]
Given $A$ for the harmonic oscillator, it is straightforward to compute the state transition matrices in \eqref{eqn:PHO_Phi}. To find $J_\varphi(0)$ we need to solve Equation \eqref{eqn:y} where
\begin{equation*}
\widetilde{A} = \left[\begin{array}{cc;{3pt/3pt}cc}
0 & 1 \ \ & 0 & 0 \\
-\omega_0^2 & 0 \ \ & 0 & 0 \\[1mm] \hdashline[3pt/3pt] \\[-4mm]
0 & 0 \ \ & 0 & 1 \\
0 & 0 \ \ & -\omega_0^2 & 0
\end{array}\right] \ \
\widetilde{b} = \left[\begin{array}{c}
0 \\
{\sin(\omega_0t)}/{\omega_0} \\[1mm] \hdashline[3pt/3pt] \\[-4mm]
0 \\
-\cos(\omega_0t)
\end{array}\right].
\end{equation*}
Using Equation \eqref{eqn:y},
\begin{equation*}
y(2\pi) = \left[\begin{array}{c}
y_1(2\pi) \\[1mm] \hdashline[3pt/3pt] \\[-4mm]
y_2(2\pi)
\end{array}\right] = \bigint_0^{2\pi}
\left[\begin{array}{c;{3pt/3pt}c}
e^{A(2\pi-\tau)} & \mathbf{0}_{2\times2} \\[1mm] \hdashline[3pt/3pt] \\[-4mm]
\mathbf{0}_{2\times2} & e^{A(2\pi-\tau)}
\end{array}\right]
\left[\begin{array}{c}
0 \\
\sin(\omega_0\tau)/\omega_0 \\[1mm] \hdashline[3pt/3pt] \\[-4mm]
0 \\
-\cos(\omega_0\tau)
\end{array}\right]
\,d\tau
= \left[\begin{array}{c}
-\pi/\omega_0^2 \\
0 \\[1mm] \hdashline[3pt/3pt]\\[-4mm]
0 \\
-\pi
\end{array}\right]
.
\end{equation*}
As in Lemma \ref{lem:PDE} this gives us
\begin{equation*}
J_\varphi(0)  = \left[\begin{array}{c;{3pt/3pt}c}
y_1(2\pi)\ \ & \ y_2(2\pi)
\end{array}\right] =
\left[\begin{array}{c;{3pt/3pt}c}
    -\pi/\omega_0^2\ \ & 0 \\[1mm]
    0\ \ & \ -\pi
\end{array}\right]
\end{equation*}
and the expression for the inverse Jacobian in \eqref{eqn:PHO_Phi} follows.
\proofbox

\noindent
\textbf{\hypertarget{cor:PHO1_proof}{Proof of Corollary \ref{cor:projA_PHO}}}\\[2mm]
Recall the results from Lemma~\ref{lem:PHO}.  By direct substitution of the state transition matrix and Jacobian into~\eqref{eqn:projA_spec}, we get
\begin{align*}
P_{{\cal{A}}_{0,0}}(u^-)(t) &= u^-(t)+\begin{bmatrix}
0 & 1
\end{bmatrix}
\begin{bmatrix}
\cos(\omega_0t) & \omega_0\sin(\omega_0t) \\ -\ds\frac{\sin(\omega_0t)}{\omega_0} & \cos(\omega_0t)
\end{bmatrix}
\begin{bmatrix}
-\frac{\omega_0^2}{\pi} & 0 \\ 0 & -\ds\frac{1}{\pi}
\end{bmatrix}
\\& \left(\begin{bmatrix}
\cos(2\pi\omega_0) & \ds\frac{\sin(2\pi\omega_0)}{\omega_0} \\ -\omega_0\sin(2\pi\omega_0) & \cos(2\pi\omega_0)
\end{bmatrix}\begin{bmatrix}
s_0 \\ v_0
\end{bmatrix}\right. \\& +\left.\int_0^{2\pi}\begin{bmatrix}
\cos(\omega_0(2\pi-\tau)) & \ds\frac{\sin(\omega_0(2\pi-\tau))}{\omega_0} \\ -\omega_0\sin(\omega_0(2\pi-\tau)) & \cos(\omega_0(2\pi-\tau))
\end{bmatrix}\begin{bmatrix}
0 \\ 1
\end{bmatrix}u^-(\tau)\,d\tau-\begin{bmatrix}
s_f \\ v_f
\end{bmatrix}\right), \\
&= u^-(t)+\begin{bmatrix}
\ds\frac{\omega_0\sin(\omega_0t)}{\pi} \\ \ds-\frac{\cos(\omega_0t)}{\pi}\end{bmatrix}^T\left(\begin{bmatrix}
s_0-s_f \\ v_0-v_f
\end{bmatrix}\int_0^{2\pi}\begin{bmatrix}
-u^-(\tau)\ds\frac{\sin(\omega_0\tau)}{\omega_0} \\ u^-(\tau)\cos(\omega_0\tau)
\end{bmatrix}d\tau\right), \\
&= u^-(t)+\frac{\omega_0}{\pi}\bigg(s_0-s_f-\frac{1}{\omega_0}\int_0^{2\pi}\sin(\omega_0\tau)u^-(\tau)\,d\tau\bigg)\sin(\omega_0t) \\& -\frac{1}{\pi}\bigg(v_0-v_f+\int_0^{2\pi}\cos(\omega_0\tau)u^-(\tau)\,d\tau\bigg)\cos(\omega_0t),
\end{align*}
as stated.
\proofbox

\noindent
\textbf{\hypertarget{lem:PDHO1_proof}{Proof of Lemma \ref{lem:PDHO_crit}}}\\[2mm]
Given $A$ for the damped harmonic oscillator with $\zeta=1$, it is straightforward to compute the state transition matrix in \eqref{eqn:PDHO1_Phi}.
To find $J_\varphi(0)$ we need to solve Equation \eqref{eqn:y} with
\begin{equation*}
\widetilde{A} = \left[\begin{array}{cc;{3pt/3pt}cc}
0 & 1 \ \ & 0 & 0 \\
-\omega_0^2 & -2\omega_0 \ \ & 0 & 0 \\[1mm] \hdashline[3pt/3pt] \\[-4mm]
0 & 0 \ \ & 0 & 1 \\
0 & 0 \ \ & -\omega_0^2 & -2\omega_0
\end{array}\right], \ \
\widetilde{b} = \left[\begin{array}{c}
0 \\
te^{\omega_0t} \\[1mm] \hdashline[3pt/3pt] \\[-4mm]
0 \\
-e^{\omega_0t}(\omega_0t+1)
\end{array}\right].
\end{equation*}
Using Equation \eqref{eqn:y}
\begin{align*}
y(2\pi) &= \left[\begin{array}{c}
y_1(2\pi) \\[1mm] \hdashline[3pt/3pt] \\[-4mm]
y_2(2\pi)
\end{array}\right] = \bigint_0^{2\pi}
\left[\begin{array}{c;{3pt/3pt}c}
e^{A(2\pi-\tau)} & \mathbf{0}_{2\times2} \\[1mm] \hdashline[3pt/3pt] \\[-4mm]
\mathbf{0}_{2\times2} & e^{A(2\pi-\tau)}
\end{array}\right]
\left[\begin{array}{c}
0 \\
\tau e^{\omega_0\tau} \\[1mm] \hdashline[3pt/3pt] \\[-4mm]
0 \\
-e^{\omega_0\tau}(\omega_0\tau+1)
\end{array}\right]
\,d\tau 
\end{align*}
yields, after integration, the required expression in~\eqref{eqn:y_PDHO1}. So in the same way as in the proofs of Lemmas~\ref{lem:PDE}--\ref{lem:PHO}
\begin{align*}
J_\varphi(0) &= \left[\begin{array}{c;{3pt/3pt}c}
y_1(2\pi)\ \ & \ y_2(2\pi)
\end{array}\right],
\end{align*}
and once inverted the expression in \eqref{eqn:Jphi_PDHO} follows.
\proofbox

\noindent
\textbf{\hypertarget{cor:PDHO1_proof}{Proof of Corollary \ref{cor:projA_PDHO1}}}\\[2mm]
We begin by finding $x(2\pi)$.
\begin{align*}
\begin{bmatrix}
x_1(2\pi) \\ x_2(2\pi)
\end{bmatrix} &= e^{-2\pi\omega_0}\begin{bmatrix}
2\pi\omega_0+1 & 2\pi \\ -2\pi\omega_0^2 & -2\pi\omega_0+1
\end{bmatrix}\begin{bmatrix}
s_0 \\ v_0
\end{bmatrix}\\ 
&\ \ \ +\int_0^{2\pi} e^{-(2\pi-\tau)\omega_0}\begin{bmatrix}
(2\pi-\tau)\omega_0+1 & (2\pi-\tau) \\ -(2\pi-\tau)\omega_0^2 & -(2\pi-\tau)\omega_0+1
\end{bmatrix}\begin{bmatrix}
0 \\ 1
\end{bmatrix}u^-(\tau)\,d\tau, \\
&= e^{-2\pi\omega_0}\begin{bmatrix}
s_0(2\pi\omega_0+1)+2\pi v_0+\int_0^{2\pi} u^-(\tau)(2\pi-\tau)e^{\omega_0\tau}\, d\tau \\
-2\pi s_0\omega_0^2 -v_0(2\pi\omega_0-1)-\int_0^{2\pi} u^-(\tau)((2\pi-\tau)\omega_0-1)e^{\omega_0\tau}\, d\tau
\end{bmatrix}.
\end{align*}

Recall the results from Lemma~\ref{lem:PDHO_crit}. By direct substitution of the state transition matrix and Jacobian into~\eqref{eqn:projA_spec}, we get
{\small\begin{align*}
P_{{\cal{A}}_{\omega_0,1}}(u^-)(t) &= u^-(t)+\begin{bmatrix}
0 & 1
\end{bmatrix}e^{\omega_0t}\!\begin{bmatrix}
-\omega_0t+1 & t\omega_0^2 \\ -t & \omega_0t+1
\end{bmatrix}\!\frac{1}{y_{11}(2\pi)y_{22}(2\pi)-y_{12}(2\pi)y_{21}(2\pi)} \\ &\times \begin{bmatrix}
y_{22}(2\pi) & -y_{21}(2\pi) \\ -y_{12}(2\pi) & y_{11}(2\pi)
\end{bmatrix}\left(\begin{bmatrix}
x_1(2\pi) \\ x_2(2\pi)
\end{bmatrix}-\begin{bmatrix}
s_f \\ v_f
\end{bmatrix}
\right).
\end{align*}}
After carrying out some of the matrix multiplications on the right-hand side one gets
{\small\begin{align*}
P_{{\cal{A}}_{\omega_0,1}}(u^-)(t) & = u^-(t)+\frac{e^{\omega_0(t-2\pi)}}{y_{11}y_{22}-y_{12}y_{21}}\begin{bmatrix}
-y_{22}t - y_{12}(\omega_0 t+1) \\ -y_{21}t + y_{11}(\omega_0 t+1)
\end{bmatrix}^T \begin{bmatrix}
x_1(2\pi) - s_f \\ x_2(2\pi) - v_f
\end{bmatrix},
\end{align*}}
where we have omitted the arguments for $y$ to save space. Expansions and further manipulations yield the required expression.
\proofbox

\noindent
\textbf{\hypertarget{lem:PDHO2_proof}{Proof of Lemma~\ref{lem:PDHO_2}}}\\[2mm]
Given $A$ for the damped harmonic oscillator with $\zeta>1$, it is straightforward to compute the state transition matrices in \eqref{eqn:PDHO2_Phi}. To find $J_\varphi(0)$ we need to solve Equation \eqref{eqn:y} where
\begin{equation*}
\widetilde{A} = \left[\begin{array}{cc;{3pt/3pt}cc}
0 & 1 \ \ & 0 & 0 \\
-\omega_0^2 & -2\omega_0\zeta \ \ & 0 & 0 \\[1mm] \hdashline[3pt/3pt] \\[-4mm]
0 & 0 \ \ & 0 & 1 \\
0 & 0 \ \ & -\omega_0^2 & -2\omega_0\zeta
\end{array}\right] \ \
\widetilde{b} = \frac{e^{\alpha t}}{\beta}\left[\begin{array}{c}
0 \\
\sinh(\beta t) \\[1mm] \hdashline[3pt/3pt] \\[-4mm]
0 \\
-\omega_0\sinh(\beta t+\eta)
\end{array}\right].
\end{equation*}
Using Equation \eqref{eqn:y}
\begin{align*}
y(2\pi) &= \frac{e^{\alpha\tau}}{\beta}\bigint_0^{2\pi} 
\left[\begin{array}{c;{3pt/3pt}c}
e^{A(2\pi-\tau)} & \mathbf{0}_{2\times2} \\[1mm] \hdashline[3pt/3pt] \\[-4mm]
\mathbf{0}_{2\times2} & e^{A(2\pi-\tau)}
\end{array}\right]
\left[\begin{array}{c}
0 \\
e^{\alpha \tau}\sinh(\beta \tau) \\[1mm] \hdashline[3pt/3pt] \\[-4mm]
0 \\
-\omega_0\sinh(\beta\tau+\eta)
\end{array}\right]
\,d\tau
\end{align*}
which, after integration, yields \eqref{eqn:y_PDHO2}. As in Lemmas \ref{lem:PDE}--\ref{lem:PDHO_crit}
\[J_\varphi(0) = \left[\begin{array}{c;{3pt/3pt}c}
y_1(2\pi)\ \ & \ y_2(2\pi)
\end{array}\right],\]
inversion of which results in the expression in \eqref{eqn:Jphi_PDHO}.
\proofbox

\newpage
\noindent
\textbf{\hypertarget{cor:PDHO2_proof}{Proof of Corollary \ref{cor:projA_PDHO2}}}\\[2mm]
We begin by finding $x(t_f)$.
\begin{align*}
x(2\pi) &= \dfrac{e^{-2\pi\alpha}}{\beta}\begin{bmatrix}
\omega_0\sinh(2\pi\beta+\eta) & \sinh(2\pi\beta) \\ -\omega_0^2\sinh(2\pi\beta) & \omega_0\sinh(-2\pi\beta+\eta)
\end{bmatrix}\begin{bmatrix}
s_0 \\ v_0
\end{bmatrix} \\ & +\bigint_0^{2\pi}\dfrac{e^{-\alpha(2\pi-\tau)}}{\beta}\begin{bmatrix}
u^-(\tau)\sinh(\beta(2\pi-\tau)) \\ u^-(\tau)\omega_0\sinh(-\beta(2\pi-\tau)+\eta)
\end{bmatrix}\,d\tau, \\
&= \dfrac{e^{-\alpha(2\pi-\tau)}}{\beta}\begin{bmatrix}
s_0\omega_0\sinh(2\pi\beta+\eta)+v_0\sinh(2\pi\beta) + C \\
-s_0\omega_0\sinh(2\pi\beta)+v_0\omega_0\sinh(-2\pi\beta+\eta)+D
\end{bmatrix},
\end{align*}
where
\begin{align*}
C &:= \int_0^{2\pi}u^-(\tau)e^{\alpha\tau}\sinh(\beta(2\pi-\tau))\,d\tau , \\
D &:= \int_0^{2\pi}e^{\alpha\tau}u^-(\tau)\omega_0\sinh(-\beta(2\pi-\tau)+\eta)\,d\tau.
\end{align*}
Now, by direct substitution of the state transition matrices and Jacobian from Lemma~\ref{lem:PDHO_2} into Equation \eqref{eqn:projA_spec}
\begin{align*}
P_{{\cal A}_{\omega_0,\zeta}}(u^-)(t) &= u^-(t)+\dfrac{e^{\alpha (t-2\pi)}}{\beta^2(y_{11}y_{22}-y_{12}y_{21})}\\&\times\begin{bmatrix}
-y_{22}\sinh(\beta t)-y_{12}\omega_0\sinh(\beta t+\eta) \\ y_{21}\sinh(\beta t)+y_{11}\omega_0\sinh(\beta t+\eta)
\end{bmatrix}^T\begin{bmatrix}
x_1-\dfrac{\beta s_f}{e^{-2\pi\alpha}} \\ x_2-\dfrac{\beta v_f}{e^{-2\pi\alpha}}
\end{bmatrix}, \\
&=u^-(t)+ \dfrac{e^{\alpha (t-2\pi)}}{\beta^2(y_{11}y_{22}-y_{12}y_{21})}\bigg(\!-(y_{22}\sinh(\beta t)+y_{12}\omega_0\sinh(\beta t+\eta))\\ &\!\times\!\left(\!x_1-\dfrac{\beta s_f}{e^{-2\pi\alpha}}\!\right)\!+\!(y_{21}\sinh(\beta t)+y_{11}\omega_0\sinh(\beta t+\eta))\left(x_2-\dfrac{\beta v_f}{e^{-2\pi\alpha}}\right)\!\!\bigg).
\end{align*}
where $x_i$, $y_{ij}$, $i,j = 1,2$, are all evaluated at $2\pi$, but not shown for clarity.
\proofbox

\noindent
\textbf{\hypertarget{lem:PDHO3_proof}{Proof of Lemma \ref{lem:PDHO_3}}}\\[2mm]
Given $A$ for the damped harmonic oscillator with $0<\zeta<1$ we compute the state transition matrices in \eqref{eqn:PDHO3_Phi}. To find $J_\varphi(0)$ we must solve Equation \eqref{eqn:y} where
\begin{equation*}
\widetilde{A} = \left[\begin{array}{cc;{3pt/3pt}cc}
0 & 1 \ \ & 0 & 0 \\
-\omega_0^2 & -2\omega_0\zeta \ \ & 0 & 0 \\[1mm] \hdashline[3pt/3pt] \\[-4mm]
0 & 0 \ \ & 0 & 1 \\
0 & 0 \ \ & -\omega_0^2 & -2\omega_0\zeta
\end{array}\right] \ \
\widetilde{b} = \frac{e^{\alpha t}}{\tb}\left[\begin{array}{c}
0 \\ \sin(\tb t) \\[1mm] \hdashline[3pt/3pt] \\[-4mm]
0 \\ -\omega_0\cos(\tb t+\gamma)
\end{array}\right].
\end{equation*}
Then using Equation \eqref{eqn:y}
\begin{align*}
y(2\pi) &= \bigint_0^{2\pi}\frac{e^{\alpha\tau}}{\tb}\left[\begin{array}{c;{3pt/3pt}c}
e^{A(2\pi-\tau)} & \mathbf{0}_{2\times2} \\[1mm] \hdashline[3pt/3pt] \\[-4mm]
\mathbf{0}_{2\times2} & e^{A(2\pi-\tau)}
\end{array}\right]\left[\begin{array}{c}
0 \\ \sin(\tb \tau) \\[1mm] \hdashline[3pt/3pt] \\[-4mm]
0 \\ -\omega_0\cos(\tb t+\gamma)
\end{array}\right]\,d\tau\,.
\end{align*}
After integration, we have expression \eqref{eqn:y_PDHO3}. Recall (as in Lemmas \ref{lem:PDE}--\ref{lem:PDHO_2})
\[J_\varphi(0) = \left[\begin{array}{c;{3pt/3pt}c}
y_1(2\pi)\ \ & \ y_2(2\pi)
\end{array}\right],\]
which after inverting yields the expression in \eqref{eqn:Jphi_PDHO}.
\proofbox

\noindent
\textbf{\hypertarget{cor:PDHO3_proof}{Proof of Corollary \ref{cor:projA_PDHO3}}}\\[2mm]
We begin by computing $x(2\pi)$.
\begin{align*}
x(2\pi) &= \dfrac{e^{-2\pi\alpha}}{\tb}\begin{bmatrix}
\omega_0\cos(2\pi\tb+\gamma) & \sin(2\pi\tb) \\ -\omega_0^2\sin(2\pi\tb) & \omega_0\cos(2\pi\tb-\gamma)
\end{bmatrix}\begin{bmatrix}
s_0 \\ v_0
\end{bmatrix}+\bigint_0^{2\pi}\dfrac{e^{-(2\pi-\tau)\alpha}}{\tb}\\&\times\begin{bmatrix}
\omega_0\cos((2\pi-\tau)\tb+\gamma) & \sin((2\pi-\tau)\tb) \\ -\omega_0^2\sin((2\pi-\tau)\tb) & \omega_0\cos((2\pi-\tau)\tb-\gamma)
\end{bmatrix}\begin{bmatrix}
0 \\ u^-(\tau)
\end{bmatrix}\,d\tau \\
& = \dfrac{e^{-2\pi\alpha}}{\tb}\begin{bmatrix}
s_0\omega_0\cos(2\pi\tb+\gamma)+v_0\sin(2\pi\tb)+C \\
-s_0\omega_0^2\sin(2\pi\tb)+v_0\omega_0\cos(2\pi\tb-\gamma)+D
\end{bmatrix}
\end{align*}
where
\begin{align*}
C&:=\int_0^{2\pi}e^{\alpha\tau}\sin(\tb(2\pi-\tau))u^-(\tau)\,d\tau,\\ D&:=\int_0^{2\pi}e^{\alpha\tau}\omega_0\cos((2\pi-\tau)\tb+\gamma)u^-(\tau)\,d\tau.
\end{align*}
In the following we omit the arguments of $x$ and $y$ to save space. After direct substitution of the results from Lemma \ref{lem:PDHO_3} into Equation \eqref{eqn:projA_spec} and simple matrix multiplication one finds
\begin{align*}
P_{{\cal A}_{\omega_0,\zeta}}(u^-)&(t) = u^-(t)+\dfrac{e^{\alpha t}}{\tb^2(y_{11}y_{22}-y_{12}y_{21})}\\&\times\begin{bmatrix}
-y_{22}\sin(\tb t)-y_{12}\omega_0\cos(\tb t+\gamma) \\ y_{21}\sin(\tb t)+y_{11}\omega_0\cos(\tb t+\gamma)
\end{bmatrix}^T\begin{bmatrix}
x_1e^{2\pi\alpha}-\tb s_f \\ x_2e^{2\pi\alpha}-\tb v_f
\end{bmatrix},\\
&= u^-(t)+\dfrac{e^{\alpha(t-2\pi)}}{\tb^2(y_{11}y_{22}-y_{12}y_{21})}\Bigg(\!\!-\!\left(y_{22}\sin(\tb t)+y_{12}\omega_0\cos(\tb t+\gamma)\right)\\&\!\times\!\left(\!x_1-\frac{\tb s_f}{e^{-2\pi\alpha}}\!\right) \!+\! \left(y_{21}\sin(\tb t)+y_{11}\omega_0\cos(\tb t+\gamma)\right)\!\!\left(x_2-\frac{\tb v_f}{e^{-2\pi\alpha}}\right)\!\!\Bigg).
\end{align*}
\proofbox

\ \\
\noindent
{\bf\Large Data Availability} \\[2mm]
The full resolution Matlab graph/plot files that support the findings of this study are available from the corresponding author upon request.

\ \\
\noindent
{\bf\Large Conflict of Interest} \\[2mm]
The authors have no competing, or conflict of, interests to declare that are relevant to the content of this article.

\section*{Acknowledgments}
The authors offer their warm thanks to William Hager who made useful comments on an earlier preprint version of their paper in~\cite{BurCalKay2022}.  They are grateful to Walaa Moursi for pointing to a specific result in~\cite{BauCombettes} about convergence of the Douglas--Rachford algorithm. BIC was supported by an Australian Government Research Training Program Scholarship. No funding was received by RSB and CYK to assist with the preparation of this manuscript.

\end{document}